\documentclass[10pt]{amsart}
\usepackage[margin=1.3in,marginparsep=0.1in]{geometry}
\usepackage{marginnote}
\usepackage{tikz-cd}
\usepackage{fancyhdr}
\usepackage[hidelinks]{hyperref}
\usepackage{setspace}
\usetikzlibrary{calc}
\theoremstyle{plain}
\newtheorem{thm}{Theorem}[section]
\newtheorem{lem}[thm]{Lemma}
\newtheorem{prop}[thm]{Proposition}
\newtheorem{si}[thm]{Situation}

\newtheorem{const}[thm]{Construction}

\theoremstyle{definition}
\newtheorem{defn}[thm]{Definition}
\theoremstyle{remark}
\newtheorem{rem}[thm]{Remark}

\theoremstyle{definition}
\newtheorem{ex}[thm]{Example}
\newtheorem{nonex}[thm]{Non-example}
\newtheorem{nt}[thm]{Notation}
\newtheorem{claim}[thm]{Claim}

\newcommand{\ZZ}{\mathbb{Z}}

\newcommand{\la}{\lambda}

\newcommand{\ti}{\tilde}

\newcommand{\mO}{\mathscr{O}}
\newcommand{\mV}{\mathscr{V}}
\newcommand{\mL}{\mathscr{L}}

\newcommand{\cG}{\mathcal{G}}

\newcommand{\cI}{\mathcal{I}}

\newcommand{\ot}{\otimes}

\newcommand{\cd}{\cdot}

\let \mb=\mathbb
\let \mc= \mathcal
\let \ra=\rightarrow 
\let \sub=\subset
\let \Ga=\Gamma
\let \La=\Lambda
\let \la=\lambda
\let \sig=\sigma
\let \al=\alpha
\let \be=\beta

\let \fr=\frac
\let \ga=\gamma

\let \ot=\otimes

\let \ul=\underline

\let \pr=\prime
\let \hra=\hookrightarrow

\let \ubr=\underbrace
\let \wt=\widetilde
\let \:=\colon
\numberwithin{equation}{section}
\usepackage{amsfonts}
\newcommand{\g}{\mathfrak{g}}
\usepackage{pdflscape}
\usepackage{tikz-cd}
\usepackage{amsmath,amssymb,latexsym,amsthm}
\usepackage{mathtools}
\usepackage{graphicx}
\usepackage{mathrsfs}
\usepackage{amscd}
\usepackage{color}
\newtheorem{notation}[thm]{Notation}

\DeclareMathOperator\Pic{Pic}

\DeclareMathOperator\van{Van}

\DeclareMathOperator\mtd{md}

\DeclareMathOperator\val{val}
\DeclareMathOperator\fl{Flag}

\input xy
\xyoption{all}
\DeclareRobustCommand{\SkipTocEntry}[5]{}
\begin{document}
\title{Secant planes of a general curve via degenerations}
\author{Ethan Cotterill, Xiang He, and Naizhen Zhang}
%\maketitle
\begin{abstract}
We study linear series on a general curve of genus $g$, whose images are exceptional with respect to their secant planes. Each such exceptional secant plane is algebraically encoded by an {\it included} linear series, whose number of base points computes the incidence degree of the corresponding secant plane. With enumerative applications in mind, we construct a moduli scheme of inclusions of limit linear series with base points over families of nodal curves of compact type, which we then use to compute combinatorial formulas for the number of secant-exceptional linear series when the spaces of linear series and of inclusions are finite.
\end{abstract}
\address[Ethan Cotterill]{Instituto de Matem\'atica, Universidade Federal Fluminense, Rua Prof Waldemar de Freitas, S/N, Campus do Gragoat\'a, CEP 24.210-201, Niter\'oi, RJ, Brazil}
\email{cotterill.ethan@gmail.com}
\address[Xiang He]{Einstein Institute of Mathematics, Edmond J. Safra Campus, The Hebrew University of Jerusalem,
Givat Ram. Jerusalem, 9190401, Israel}
\email{xiang.he@mail.huji.ac.il}
\thanks{Naizhen Zhang was supported by the Methusalem Project Pure 
Mathematics at KU Leuven during the preparation of this paper. Xiang He is supported by the ERC Consolidator Grant 770922 - BirNonArchGeom.}
\address[Naizhen Zhang]{Institut f\"ur Differentialgeometrie, Leibniz Universit\"at Hannover, Welfengarten 1, 30167 Hannover, Germany}
\email{naizhen.zhang@math.uni-hannover.de}

\maketitle
\tableofcontents

\section{Introduction}\label{sec:intro}
Determining when an abstract curve $C$ comes equipped with a nondegenerate map of degree $m$ to $\mb{P}^s$ is central to curve theory. The {\it Brill--Noether theorem} of Griffiths and Harris gives a complete solution to this problem when $C$ is general in $\mc{M}_g$, but there are also many variations on this basic problem. One of these involves requiring that the image of $C$ intersect a linear subspace of prescribed dimension in the ambient $\mb{P}^s$, with prescribed incidence degree. The linear subspaces are then {\it secant planes} of the image of $C$, and are associated with inclusions of linear series of the form
\begin{equation}\label{inclusion_of_LLS}
g^{s-d+r}_{m-d}+ p_1+ \dots+ p_d \hookrightarrow g^s_m
\end{equation}
in which $d$ is the incidence degree and $(d-r-1)$ is the dimension of the secant plane. Here the points $p_1, \dots p_d$ in $C$ are base points of the included series.

\medskip
Farkas~\cite{FSec} developed a dimension-theoretic understanding of spaces of secant planes to linear series using Eisenbud--Harris' theory of limit linear series. The first author \cite{Co1,Co2} counted these secant planes.

%In \cite{Co1,Co2,FSec}, the first author and G. Farkas studied these questions using Eisenbud and Harris' theory of {\it limit linear series} for nodal curves of compact type. Farkas' contribution is to the dimension-theoretic understanding of (spaces of) secant planes of linear series, while the first author's results primarily deal with the issue of counting these.

\medskip
In the intervening ten years, a couple of interesting developments have occurred. Osserman has developed an alternative theory of limit linear series, closely related to that of Eisenbud and Harris (and indeed the definitions are set-theoretically equivalent for curves of compact type). Osserman's theory has somewhat better functorial properties than that of Eisenbud and Harris; it is also more cumbersome to use. On the other hand, it has become increasingly clear that {\it chains of elliptic curves} behave well (with respect to linear series) as surrogates for a general curve. Our aim in this paper is put both of these developments to good use, and initiate a program for counting secant planes of linear series on a general curve, by first counting inclusions of limit linear series as in \eqref{inclusion_of_LLS} and then lifting the result to a general curve via a {\it smoothing} theorem.

\medskip
To prove our smoothing theorem, we appeal to the theory of limit linear series developed by Osserman in \cite{osserman2014limit}. Although our focus is on limit linear series over curves of compact type, Osserman's (more general) theory provides a natural template for constructing proper moduli spaces, and for proving smoothing theorems.
%\begin{enumerate}
%\item[1.] It clarifies the meaning of ``base point" of a limit linear series, particularly when one is included inside another;
%\item[2.] It provides a natural template for constructing proper moduli spaces, and for proving smoothing theorems.
%\end{enumerate}
We in fact give {\it two} distinct constructions of a moduli space of inclusions of (limit) linear series; the first space is proper, while the second (which agrees set-theoretically with the first along an open locus) has a structure more amenable to local dimension estimates. The second construction makes use of {\it linked chains of flags}, which we introduce as (flagged) generalizations of the linked chains described by Murray and Osserman. Associated with linked chains of flags are {\it linked determinantal loci}, of which inclusions of limit linear series \eqref{inclusion_of_LLS} are archetypal examples.

\medskip
When it comes to {\it counting} (inclusions of) limit linear series, Eisenbud--Harris' theory of limit linear series is more combinatorially advantageous. On the other hand, it is known that the two theories (schematically) agree when we restrict ourselves to {\it refined} limit linear series, a fact that we will leverage to count $d$-secant $(d-r-1)$-planes to (the image of) a fixed linear series $g^s_m$ of well-behaved combinatorial type on a general curve when the spaces of linear series $g^s_m$ (resp, inclusions \eqref{inclusion_of_LLS} associated with a fixed $g^s_m$) are both zero-dimensional. Our method produces enumerative formulas whose structure is qualitatively different from the usual formulas obtained via classical intersection theory described in \cite{ACGH} and \cite{Co1,Co2}. Crucially, the formulas that arise from inclusions of limit linear series on elliptic chains are manifestly {\it positive}.

\subsection{Roadmap}
The plan for the remainder of this paper is as follows. In Section~\ref{sec:basic}, we review Osserman's theory of limit linear series; for the reader's convenience, in Proposition~\ref{equiv} we briefly recall the equivalence between Osserman and Eisenbud--Harris limit linear series on a fixed curve $X$ when $X$ is of compact type. Section~\ref{sec:linked} develops the theory of linked chains of flags. A key subsidiary notion is that of {\it $\ul{r}$-strictness}, in which $\ul{r}$ is a rank vector. Theorem~\ref{thm:dimension of linked determinantal loci} establishes that a linked determinantal locus has the expected dimension along its $\ul{r}$-strict locus. 

\medskip
In Section~\ref{sec:moduli}, we first provide a construction of the moduli scheme of inclusion of linear series on a non-singular curve in Proposition \ref{prop:presentation}. After giving our definition of inclusion of limit linear series in Definition \ref{defn:inc_lls}, we describe two constructions of a projective moduli scheme of inclusion of limit linear series in subsections \ref{subsubsec: scheme structure of moduli limit} and \ref{subsubsec: alternate construction}. The upshot is that the induced reduced structure of two moduli schemes agree. Using the second construction, we then manage to prove a smoothing theorem for inclusion of limit linear series in the case $\rho=\mu=0$ (Theorem \ref{thm:smoothing theorem}), provided that the ambient series has certain nice properties. This is the theoretical foundation for enumerating inclusion of linear series via degeneration.

\medskip
In Section~\ref{combinatorics_of_LLS}, we give a general algorithm for counting inclusions \eqref{inclusion_of_LLS} of limit linear series in cases where the two fundamental invariants $\rho$ (which computes the dimension of the space of $g^s_m$) and $\mu$ (which computes the dimension of the space of inclusions \eqref{inclusion_of_LLS} associated with a fixed $g^s_m$) are both equal to zero. Our results, like the classical intersection-theoretic results are most explicit when $r=1$; see Theorem~\ref{equality_of_numbers}. We introduce the notion of {\it shift poset}, which precisely dictates how base points force the included $g^{s-d+r}_m$, viewed as a point of the Grassmannian $\mbox{Gr}(s-d+r+1,s+1)$, to move. It seems reasonable to speculate that the coefficients in our formulas are related in some as-yet-unknown way to Schubert calculus.

\medskip
Finally, in Section~\ref{pathology}, we present an example due to Melody Chan, which shows that the moduli space of included limit linear series we constructed in Section~\ref{sec:moduli} may have components of unexpectedly large dimension. A deeper understanding of this phenomenon as it relates to smoothability is crucial to extending the enumerative analysis carried out in this paper to situations in which the {\it total} dimension $\rho+\mu$ of the space of linear series with secant planes is zero, but $\rho$ is strictly positive.

\subsection{Acknowledgements} We thank Melody Chan, Javier Gargiulo, Alberto L\'opez, Brian Osserman and Nathan Pflueger for useful conversations; the Brazilian CNPq, whose postdoctoral scheme allowed the first and third authors to meet; and the anonynomous referee, for his/her careful reading and corrections.

\section{Basic definitions}\label{sec:basic}
%facilitates us in proving basic results of inclusion of limit linear series, such as a smoothing theorem. We first describe our setup. 

Throughout this section, $X_0$ will denote a proper, reduced and connected nodal curve, whose irreducible components are all smooth. In particular, we do not assume $X_0$ to be of compact type a priori. Notationally, %, but we will adopt extra assumptions when applying to situations of our interest. Also, 
$d$ will always denote a positive integer, while $\Gamma$ will always denote the dual graph of $X_0$. For the sake of convenience, we assume $\Gamma$ is directed. %This is merely for the sake of convenience, 
Directedness is combinatorially required by our specifying a reference component for $X_0$. For example, when $X_0$ is of compact type, it is natural to specify that $\Gamma$ be a rooted tree, with all edges pointing away from the root. The irreducible components of $X_0$ are always denoted by $(C_v)_{v\in V(\Gamma)}$, unless stated otherwise.

The main premise of Osserman's theory of limit linear series is that one should consider all possible multidegrees of line bundles on a nodal curve, and not just those {\it concentrated} multidegrees considered in the classical theory of Eisenbud and Harris.\footnote{For example, if $X_0$ is compact type, a line bundle on $X_0$ with concentrated mutidegree restricts to a degree zero line bundle over all but one component.} While this extra freedom comes at the expense of additional complexity, it will turn out to be crucial in studying inclusions of limit linear series with base points.  

We now review the basic ingredients of Osserman's theory.
\begin{defn}
The \textit{multidegree} of a line bundle $\mL$ over $X_0$, $\mtd(\mL)$, is the integer-valued function $V(\Gamma)\to\ZZ:v\mapsto \deg(\mL|_{C_v})$. The \textit{total degree} of $\mL$ is then $\sum_{v\in V(\Gamma)}\mtd(\mL)(v)$. 
\end{defn}

Distinct line bundles of the same total degree may be limits on the special fiber of a single line bundle on the generic fiber. To systematize how line bundles of the same total degree relate to one another in this way, we start by fixing a choice of {\it pseudo-divisors} $\{(\mO_v,C_v,s_v)\}_{v\in V(\Gamma)}$ over $X_0$. Here $\mO_v$ is an invertible sheaf on the nodal curve $X_0$, $C_v$ is the irreducible component of $X_0$ corresponding to the vertex $v$ in the dual graph $\Gamma$, and $s_v$ is a section of $\mO_v$. They satisfy the following conditions:
\begin{enumerate}
  \item[1.] $\mO_v|_{C_v}\cong \mO_{C_v}(-(C_v\cap C^c_v))$, $\mO_v|_{C^c_v}\cong \mO_{C^c_v}(C_v\cap C^c_v)$;
  \item[2.] $\bigotimes_{v\in V(\Gamma)}\mO_v\cong\mO_{X_0}$;
  \item[3.] $s_v\in\Gamma(X_0,\mO_v)$ is a section vanishing precisely along $C_v$.
\end{enumerate}  
%\begin{rem}
Following \cite{osserman2014limit}, we call $\{(\mO_v,C_v,s_v)\}_{v\in V(\Gamma)}$ an \textit{enriched structure} on $X_0$. 
%\end{rem}

\begin{nt}
Given $\{\mO_v\}_{v\in V(\Gamma)}$, let $f_v:\mL\to\mL\otimes\mO_v$ denote the natural morphism sending $s$ to $s\otimes s_v$. More generally, for any sequence $\vec{v}:=(v_1,...,v_n)\in V(\Gamma)^{\times n}$, let $f_{\vec{v}}:\mL\to\mL\otimes\bigotimes_{j=1}^n\mO_{v_j}$ denote the morphism defined by $s\mapsto s\otimes s_{v_1}\otimes...\otimes s_{v_n}$.
\end{nt}  

\begin{defn}
Two multidegrees $\mtd_1$ and $\mtd_2$ are \textit{similar}, as denoted by $\mtd_1\sim\mtd_2$, whenever $\mtd_1-\mtd_2=\sum_{j=1}^n \mtd(\mO_{v_j})$, for some choice of $v_1,...,v_n\in V(\Gamma)$.

\medskip
Similarly, two line bundles $\mL_1,\mL_2$ are \textit{similar} whenever $\mL_2\cong\mL_1\otimes (\bigotimes_{j=1}^n\mO_{v_j})$ for some $v_1,...,v_n\in V(\Gamma)$.
\end{defn}
\begin{rem}\label{unique-path}
Given $\mL$ and a multidegree $\mtd_0$ similar to $\mtd(\mL)$, there is up to isomorphism precisely one line bundle $\mL^{\pr}$ for which $\mL^{\pr}\sim\mL$ and $\mtd(\mL')=\mtd_0$. If we further require that $\mL^{\pr}=\mL\otimes (\bigotimes_{j=1}^n\mO_{v_j})$ such that $n$ is minimal, then the twist $\mL^{\pr}$ is \textbf{unique}; see \cite[Prop. 2.12]{osserman2014limit}.
\end{rem}
\begin{nt}\label{minimal}
Let $\mtd_1$, $\mtd_2$ be two similar multi-degrees. Let $\mO_{\mtd_2-\mtd_1}$ denote the tensor product $\bigotimes_{i=1}^n\mO_i$ of line bundles $\mO_i\in\{\mO_v\}_{v\in V(\Gamma)}$ for which 
\[
\mL\ot\mO_{\mtd_2-\mtd_1}
\]
is the unique line bundle of multi-degree $\mtd_2$ similar to $\mL$ whenever $\mL$ is of multi-degree $\mtd_1$, as per Remark ~\ref{unique-path}. 
\end{nt}

\begin{nt}
Given a line bundle $\mL$ over $X_0$ and for any multidegree $\mtd_0$ similar to $\mtd(\mL)$, let $\mL(\mtd_0)$ denote the line bundle of multidegree $\mtd_0$ and similar to $\mL$. 
\end{nt}

\begin{rem}\label{morphism}
%Following Remark~\ref{unique-path}, 
Remark~\ref{unique-path} has the following useful consequence. Given two similar line bundles $\mL$ and $\mL^{\pr}$, let $n$ be the smallest integer for which there exist $v_1,...,v_n\in V(\Gamma)^{\times n}$ and 
\[
\mL^{\pr}\cong \mL\ot(\bigotimes_{j=1}^n\mO_{v_j}).
\]
There is then a well-defined morphism 
\[
\mL\to\mL^{\pr}
\]
that maps a section $s$ to $s\ot s_{v_1}\ot...\ot s_{v_n}$.
\end{rem}

\begin{defn}
Given similar line bundles $\mL$ and $\mL^{\pr}$, we shall refer to the morphism defined in \ref{morphism} as the {\it natural} morphism from $\mL$ to $\mL^{\pr}$.
\end{defn}

\begin{lem}\label{similarity_lemma}
Similarity of line bundles is an equivalence relation. Further, whenever $\mL_1$ and $\mL_2$ are similar to one another, the natural morphisms $f:\mL_1\to\mL_2$ and $g:\mL_2\to\mL_1$ satisfy%\com{A space between the two equalities? Or "$f\circ g=0 \mathrm{\ and\ } g\circ f=0$"}
\[
f\circ g=0 \text{ and } g\circ f=0.
\]
\end{lem}

\begin{proof}
The proof of Lemma~\ref{similarity_lemma} is given in \cite{osserman2014limit}, but for the convenience of the reader, we recall the basic idea here. Namely, if $\mL_2\cong \mL_1\ot\mO_v$, then 
\[
\mL_1\cong\mL_2\ot(\bigotimes_{v'\neq v}\mO_{v'})
\]
since $\bigotimes_{v\in V(\Gamma)}\mO_v\cong\mO_{X_0}$. Expanding upon this idea gives the first assertion. 

The second assertion follows from the fact that $f$ and $g$ are defined by multiplying by the sections $s_v$ given by the chosen pseudo-divisors $\{\mO_v,s_v\}$, and $s_v$ vanishes on $C_v$. 
\end{proof}

\begin{defn}
A \textit{line bundle class} of degree $d$ over $X_0$ is an equivalence class of degree $d$ line bundles; we let $[\mL]$ denote the class of $\mL$. A {\it morphism} from $[\mL_1]$ to $[\mL_2]$ of multi-degree $w$ consists of a family of morphisms 
\[
\{\mL_{w_1}\to \mL_{w_2} \}_{w_1,w_2}
\]
such that $w_2-w_1=w$ for all $w_1,w_2$ and $\mL_{w_k}$ varies over all line bundles similar to $\mL_k$, for $k=1,2$.
\end{defn}

Clearly any given morphism of line bundles induces a morphism between the corresponding classes, for a specific pair of multidegrees. 

\medskip
Osserman's definition of limit linear series, which applies to nodal curves not-necessarily of compact type, requires introducing Osserman's notion of a \textit{concentrated multidegree}. %\com{The two notions are not equivalent in compact-type case, Osserman's concentrated degree can have negative degree on some components. I would say that they are equivalent for effective divisors.}
The latter is a generalization of the one found in Eisenbud-Harris' theory of limit linear series. In general, a concentrated multidegree may take negative values on some components. It may also be positive on more than one component if the curve is not of compact type.  Since the general definition does not play a role in our application, we refer the curious reader to \cite{osserman2014limit} for the precise definition of a concentrated multidegree as well as a brief discussion of the general situation. 

\begin{defn}\label{lls-defn}
A {\it limit $\g^r_d$ over $X_0$} consists of the following data $(\mL,(V_v)_{v\in V(\Gamma)})$.
\begin{enumerate}
  \item[1.] $\mL$ is a line bundle over $X_0$ of total degree $d$;
  \item[2.] For each $v\in V(\Gamma)$, there is some fixed choice of concentrated multidegree $\mtd_v\sim \mtd(\mL)$ such that $V_v\subset\Gamma(C_v,\mL(\mtd_v)|_{C_v})$ is an $(r+1)$-dimensional subspace of sections; 
  \item[3.] For all $\mtd_0\sim\mtd(\mL)$, the kernel of the linear map 
  \[\Gamma(X_0,\mL(\mtd_0))\to\displaystyle\bigoplus_{v\in V(\Gamma)}\Gamma (C_v,\mL(\mtd_v)|_{C_v})/V_v \]
  has dimension at least $r+1$.
  \end{enumerate}  
We also call $\mtd(\mL)$ \textit{the multidegree of the limit} $\g^r_d$.  
\end{defn}

\subsection{Equivalent formulation of the theory for curves of compact type}
In this section, we will show that the definition \eqref{lls-defn} of limit linear series given in the preceding section is in fact equivalent to the definition of Eisenbud--Harris whenever the underlying curve $X_0$ is of compact type. This will be useful for the enumerative applications to come later.

\begin{nt}\label{close_node}
Suppose $X_0$ is a curve of compact type and $v_0$ is the vertex in $\Ga=\Ga(X_0)$ corresponding to the base component which we choose from the outset.\footnote{See the discussion at the beginning of this section.} For all $v\neq v_0$, let $P^*_v$ denote the unique node on the component $C_v$ that corresponds to the unique incoming edge at $v\in\Gamma$. \footnote{More concretely, let $\ell$ denote the unique path in the directed dual graph $\Gamma$ from $v_0$ to $v$ and let $X_{\ell}$ denote the sub-curve of $X_0$ corresponding to $\ell$; then $P^*_v$ is the intersection of $C_v$ with the closure of its complement
inside $X_{\ell}$.}
\end{nt}

\begin{lem}\label{deg-sim}
When $X_0$ is of compact type, any two multidegrees $\mtd_1,\mtd_2$ of the same total degree $d$ are similar. 
\end{lem}
\begin{proof}
Any given multidegree of total degree on $X_0$ may be encoded as a {\it divisor} on the graph $\Ga$, in the sense of \cite{baker2007riemann}. The desired result then follows from the fact that similarity of multidegrees corresponds to linear equivalence of divisors on $\Ga$. But $\Ga$ is a tree, so any two divisors of $\Ga$ are linearly equivalent.

Alternatively, we may argue directly as follows (according to the dictionary of the preceding paragraph, this amounts to constructing an explicit linear equivalence of graphical divisors). Recall that $\Gamma$ is in fact a {\it directed} tree, with all edges pointing away from its root, $v_0$. It then suffices to show that an arbitrary multidegree $\mtd^{\pr}$ of total degree $d$ is similar to the concentrated multidegree $\mtd_0$, for which $\mtd_0(v_0)=d$ and $\mtd_0(v)=0$ for $v\neq v_0$. 

To do so, we make use of the enriched structure of $X_0$. Note that for $v\in V(\Gamma)$, $\mtd(\mO_v)$ takes value $-\val(v)$ at $v$, 1 at any $v^{\pr}$ adjacent to $v$, and 0 elsewhere. In particular, if $v$ is a leaf of $\Gamma$, it takes value  $-1$ at $v$ and 1 at the unique vertex $v^{\pr}$ adjacent to $v$.

On the other hand, for every vertex $v$, there is a unique directed path from $v_0$ to $v$ in $\Ga$. Let $d(v)$ be the number of edges on this path. The desired result follows by downward induction on $d(v)$. Namely, suppose $v_1,\dots,v_n$ are the vertices realizing the maximal $d(v)$ among all $v\in V(\Gamma)$, and let $v^{\pr}_j$ be the vertex adjacent to $v_j$. It is easy to see that for some choice of $m_j,n_j\ge 0$, 
\[
\mtd^{\pr}+\sum_j(n_j\mtd(\mO_{v_j})+m_j\mtd(\mO_{v^{\pr}_j}))
\]
has value zero at all $v_j$. In general, suppose $\mtd^{\pr}$ has value zero at all $v$ such that $d(v)>N$. There then exists some choice of $n_v\ge 0$ such that $\mtd^{\pr}+\sum_{v:d(v)\ge N-1}n_v\cd\mtd(\mO_v)$ has value zero at all $v$ with $d(v)\ge N$.
\end{proof}

\begin{rem}
Lemma~\ref{deg-sim} is clearly false for arbitrary nodal curves, as graphs with cycles have nontrivial Jacobians in general. %For instance, consider a curve $X_0$, with components $C_1,C_2,C_3$ such that $C_1$ meets $C_2$ at two points and $C_2$ meets $C_3$ at another two points. Then, $\mtd(\mO_{v_1})=(-2,2,0)$, $\mtd(\mO_{v_2})=(2,-4,2)$, $\mtd(\mO_{v_3})=(0,2,-2)$ and the statement fails for parity reason.  
\end{rem}

%Consequently, when $X_0$ is of compact type, we say a multidegree $\mtd$ is \textit{concentrated at} $v$ only when $\mtd(v)=d$ and $\mtd(v')=0$ for $v'\neq v$.

We now recall the Eisenbud-Harris definition of limit linear series for curves of compact type. 
\begin{defn}\label{EH-defn}
Assume $X_0$ is compact type. A {\it limit $\g^r_d$ over $X_0$} consists of the data $(\mL_v,V_v)_{v\in V(\Gamma)}$, where each $(\mL_v,V_v)$ is a $\g^r_d$ over the smooth curve $C_v$, satisfying the following compatibility condition for vanishing sequences at points of intersection $C_v\cap C_{v^{\pr}}=P$ of smooth components. Namely, when writing the vanishing sequence $\van_P(V_v)$ (resp. $\van_P(V_{v^{\pr}})$) as an increasing (resp. decreasing) sequence, we have
\[\van_P(V_v)+\van_P(V_{v'})\ge (d,d,...,d).\] 
\end{defn}

An important point is that for curves of compact type, the definitions \ref{EH-defn} and \ref{lls-defn} agree with one another.

\begin{prop}\label{equiv}
Let $X_0$ be a nodal curve of compact type. Then \ref{EH-defn} and \ref{lls-defn} are equivalent for any fixed multidegree. 
\end{prop} 

\begin{proof}
This is a special case of Theorem 5.9 in \cite{osserman2014limit}, where Osserman considered curves of \textit{pseudo-compact type}, a class of curves strictly containing curves of compact type.
\end{proof}
\subsection{Subbundles of push-forwards of vector bundles}
In order to facilitate the discussion of (limit) linear series in families, Osserman introduced the following notion of subbundles in \cite{osserman2014higherlimit}:
\begin{defn}\label{defn:subbundles}
Let $\pi:X\to B$ be a proper morphism that is locally of finite presentation, and let $E$ be a quasicoherent sheaf on
$X$, locally finitely presented and flat over $B$. A subsheaf $V$ is a {\it subbundle} of $\pi_*E$ if $V$ is locally free of finite rank, and for any $S\to B$, the natural pullback map $V_S\to \pi_{S*}E_S$ is injective.
\end{defn}
\begin{rem}\label{rem:subbundle}
It is straightforward to see that the map $V_S\to \pi_{S*}E_S$ realizes $V_S$ as a subbundle of $\pi_{S*}E_S$. 
\end{rem}
\section{Linked chains of flags}\label{sec:linked}
In this section we introduce {\it linked chains of flags}, which generalize the linked chains described in \cite{murray2016linked}. We give a dimension estimate for the linked determinantal locus associated to a linked chain of flags, which we will apply later to calculate the dimension of the moduli space of inclusion of (limit) linear series that we construct in Section~\ref{sec:moduli}; see also the proof of Theorem~\ref{thm:smoothing theorem}.

The basic set-up is as follows. Let $S$ be a scheme, and fix a choice of positive integers $d$ and $n$. Suppose that $\mathscr E_1,...,\mathscr E_n$ are vector bundles of rank $d$ on $S$ and that we are given morphisms 
$$f_i\colon \mathscr E_i\rightarrow\mathscr E_{i+1},
\ \ f^i\colon\mathscr E_{i+1}\rightarrow\mathscr E_i$$ for each $i=1,...,n-1.$ For each index $i=1,\dots,n$, suppose moreover that 
\[
\mathscr E^m_i\hookrightarrow\mathscr E^{m-1}_i\hookrightarrow\cdots\hookrightarrow\mathscr E^1_i=\mathscr E_i
\]
is a flag of subbundles of $\mathscr E_i$.

\begin{defn}\label{defn:linked chain of flags}
Given a sequence of numbers $d=d_1>\cdots>d_m>0$, 
and $s\in H^0(S,\mathscr O_S)$, we say that $( \{\mathscr E^j_\bullet\}_{1\leq j\leq m},f_\bullet,f^\bullet)$ is an \textbf{$s$-linked chain of flags} of index $(d_1,...,d_m)$ if

\begin{enumerate}
\item $\mbox{rank}(\mathscr E^j_i)=d_j$ for all $i,j$;
\item $$f_i(\mathscr E^j_i)\subset \mathscr E^j_{i+1} \mathrm{\ and\ } f^i(\mathscr E^j_{i+1})\subset \mathscr E^j_i$$ for every $1\leq i \leq n-1$ and $1\leq j\leq m$; and
\item the tuple $(\mathscr E^j_\bullet, f_\bullet, f^\bullet)$ is an $s$-linked chain in the sense of \cite{murray2016linked}.
\end{enumerate}
More precisely, for each $j$, when restricted to $(\mathscr E^j_\bullet)$, we have that (i) both $f_i\circ f^i$ and $f^i\circ f_i$ are scaling by $s$; (ii) $\ker f^i=\mbox{Im} f_i$ and $\ker f_i=\mbox{Im} f^i$ along the vanishing locus of $s$; and (iii) $\mbox{Im}f_i\cap \ker f_{i+1}=0$ and $\mbox{Im}f^{i+1}\cap\ker f^i=0$ along the vanishing locus of $s$.

%(2) At each point $x\in S$ we have inclusion of fibers $(\mathscr E^j_\bullet)_x\hookrightarrow(\mathscr E_\bullet)_x$ for each $j$.

\end{defn}

Now let $(\{\mathscr E^j_\bullet\}_{1\leq j\leq m},f_\bullet,f^\bullet)$ be an $s$-linked chain of flags of index $(d_1,...,d_m)$. Suppose that for the two extremal values $i=1,n$ we have (a compatible system of) morphisms 
\[
\begin{tikzcd} \mathscr E^m_i\dar{g^m_i} \rar& \mathscr E^{m-1}_i\rar\dar{g^{m-1}_i} &\cdots \dar{ }\rar&\mathscr E^1_i\dar{g^1_i}\\ 
\mathscr G^m_i\rar&\mathscr G^{m-1}_i\rar &\cdots\rar &\mathscr G^1 _i
\end{tikzcd}
\]
in which $\mathscr G^j_i$ is a vector bundle of rank $d_j-r^j_i$ for every $1\leq j\leq m$. 

\begin{defn}\label{defn:linked determinantal locus}
Let $\mathscr G^j_i$ and $g^j_i$ be as above, and let $r_1,...,r_m$ be a (non-strictly) decreasing sequence of nonnegative integers. The \textbf{linked determinantal locus} associated to the $s$-linked chain of flags $( \{\mathscr E^j_\bullet\}_{1\leq j\leq m},f_\bullet,f^\bullet)$ is the closed subscheme of $S$ along which the rank of the induced morphisms 
$$\pi^j_i\colon \mathscr E^j_i\rightarrow \mathscr G^j_1\oplus \mathscr G^j_n$$ is at most $d_j-r_j$ for all $1\leq j\leq m$ and $1\leq i\leq n$. In other words, it is the intersection of the $r_j$-th vanishing loci of the $\pi^j_i$ as $i$ and $j$ are allowed to vary.
\end{defn}

By the {\it $r_j$-th vanishing locus} of $\pi^j_i$ we mean the subscheme of $S$ cut out by all $(d_j-r_j+1)$-minors of $\pi^j_i$; see \cite[Definition B.1.1]{osserman2014higherlimit} for a formal definition. 

\begin{defn}\label{defn:strict linked chain of flags}
Given a tuple $\underline r=(r_m,...,r_1)$ of nonnegative integers as in Definition \ref{defn:linked determinantal locus}, an $s$-linked chain of flags $( \{\mathscr E^j_\bullet\}_{1\leq j\leq m},f_\bullet,f^\bullet)$, and $x\in S$, we set
\[
(K^j_i)_x:=\ker (f_i|_{(\mathscr E^j_i)_x})\oplus\mathrm{Im} (f_{i-1}|_{(\mathscr E^j_{i-1})_x})\mathrm{\ and \ }
(\widetilde K^j_i)_x:=\ker(\pi^j_i|_x)\cap (K^j_i)_x.
\]
We say that $( \{\mathscr E^j_\bullet\}_{1\leq j\leq m},f_\bullet,f^\bullet)$ is \textbf{$\underline r$-strict} at $x$ with respect to the maps $\pi^j_i$ of Definition~\ref{defn:linked determinantal locus} if, for all $1\leq j\leq m$ and $1\leq i\leq n$, we have
\begin{enumerate}
\item[(a)]
$(K^j_i)_x\cap (\mathscr E^l_i)_x=(K^l_i)_x$ for all $j$ and $l\geq j$; and 
\item[(b)]
$\sum_{i=1}^n\dim_{k(x)} (\widetilde K^j_i)_x \leq (n-1)r_j$ for every $2\leq j\leq m$.
\end{enumerate}
\end{defn}

Note that if $s$ is nonzero at $x$, then item (a) is satisfied automatically, and the strictness condition is equivalent to requiring that $\dim \ker(\pi^j_i|_x)\leq r_j$ for all $2\leq j\leq m$.

\begin{rem}\label{rem:strict linked chain of flags}
  For every $l< i$, set $f_{l,i}:=f_{i-1}\circ f_{i-2}\circ\cdots\circ f_l$ and $f^{l,i}:=f^l\circ f^{l+1}\circ\cdots\circ f^{i-1}$. According to Lemma 2.10 of \cite{murray2016linked}, there is a distinguished set of vectors $S^j_i\subset (\mathscr E^j_i)_x$ for which $$\mathrm{span}(S^j_i)\cap  (K^j_i)_x=\emptyset$$ and the images $\{f_{l,i}S^j_l\}_{l<i}\cup \{f^{i.l}S^j_l\}_{l>i}$ in $(\mathscr E^j_i)_x$ and $S^j_i$ are linearly independent and generate $(\mathscr E^j_l)_x$. Item (b) in the $\underline r$-strictness condition of Definition~\ref{defn:strict linked chain of flags} is used to ensure that, when $x$ is ($\underline r$-strict and) contained in the linked determinantal locus in Definition \ref{defn:linked determinantal locus},  we may similarly find, for every $1\leq j\leq m-1$, sets of vectors $\{\widetilde S^j_i\subset (\mathscr E^j_i)_x\}_{1\leq i\leq n}$ %for which $\widetilde S^j_i$ remains 
  that are disjoint with $(K^j_i)_x$, and whose images in $(\mathscr E^j_i)_x$ %(including $\widetilde S^j_i$) 
 are linearly independent and span a subspace of $\ker(\pi^j_i|_x)$ of dimension at least $r_j$. This actually follows from the fact that 
 $$\sum_{i=1}^n\dim_{k(x)}\ker(\pi^j_i|_x)\geq nr_j$$ 
 as we are free to choose $\widetilde S^j_i$ such that $\ker(\pi^j_i|_x)=\mathrm{span}\widetilde S^j_i\oplus (\widetilde K^j_i)_x$. Note that the fact that $\mathrm{span}(\widetilde S^j_i)\cap (K^j_i)_x=\emptyset$ ensures that the images of 
$ \widetilde S^j_\bullet$ in $(\mathscr E^j_i)_x$ are linearly independent. On the other hand, since 
 $$\dim_{k(x)}(\widetilde K^j_i)_x\geq \sum_{l\neq i}|\widetilde S^j_l|$$
we have that $$\sum_{i=1}^n\dim_{k(x)}(\widetilde K^j_i)_x\geq (n-1)\sum_{i=1}^n|\widetilde S^j_i|\geq (n-1)r_j.
$$
 Item (b) of Definition~\ref{defn:strict linked chain of flags} then implies that in fact $\dim_{k(x)}\ker(\pi^j_i|_x)=\sum_{i=1}^n|\widetilde S^j_i|=r_j$ for every $2\leq j\leq m.$ This along with item (a) guarantees that passing to the quotient $\{\mathscr E^j_\bullet/\ker\pi^l_\bullet\}_{1\leq j\leq l}$ induces a linked chain of flags over $\underline r$-strict points, for every $2\leq l\leq m$; see the proof of Theorem~\ref{thm:dimension of linked determinantal loci}.
\end{rem}

\begin{nonex}\label{ex:non-r-strict}
%We illustrate an example in practice which corresponds to a 
To illustrate Definition~\ref{defn:strict linked chain of flags}, we construct an example of a {\it non}-$\underline r$-strict point, as follows. Let $m=2$, $r_2=3$, and $n=4$, and let $X=X_1\cup X_2$ be the union of two rational curves glued along a point $P$. Choose parametrizations $x_{\ell}$ of $X_{\ell}$, $\ell=1,2$ such that $P$ corresponds to $\infty$ on each $X_l$, and prescribe line bundles $\mathscr L_i$ on $X$ according to their restrictions $\mathscr L_i|_{X_1}=
\mathcal O_{X_1}((4-i)P)$ and $\mathscr L_i|_{X_2}=\mathcal O_{X_2}((i-1)P)$, for $1\leq i\leq 4$. The global sections of $\mathscr L_i$ are then given by:
\[
\begin{split}
H^0(X,\mathscr L_1)&=\mathrm{span}\{(1,0),(x_1,0),(x_1^2,0),(x_1^3,1)\}; \\
H^0(X,\mathscr L_2)&=\mathrm{span}\{(1,0),(x_1,0),(x_1^2,x_2),(0,1)\}; \\
H^0(X,\mathscr L_3)&=\mathrm{span}\{(1,0),(x_1,x_2^2),(0,1),(0,x_2)\}; \text{ and} \\
H^0(X,\mathscr L_4)&=\mathrm{span}\{(1,x_2^3),(0,1),(0,x_2),(0,x_2^2)\}.
\end{split}
\]
Now let $\mathscr E^2_i$ be the pushforward of $\mathscr L_i$ to a point $x$. The (non-trivial) twisting maps between fibers of $\mathscr E^2_i$ are then given by:
\[
\begin{split}
&(1,0)\leftarrow (1,0)\leftarrow (1,0)\leftarrow (1,x_2^3); \\
&(x_1,0)\leftarrow (x_1,0)\leftarrow (x_1,x_2^2)\rightarrow (0,x_2^2); \\
&(x_1^2,0)\leftarrow (x_1^2,x_2)\rightarrow (0,x_2)\rightarrow (0,x_2); \text{ and} \\
&(x_1^3,1)\rightarrow (0,1)\rightarrow (0,1)\rightarrow (0,1).
\end{split}
\]
These maps define a 0-linked chain on $x$.
Now let $$\mathscr G^2_1=\mathscr E^2_1/\mathrm{span}\{(1,0),(x_1,0),(x_1^3,1)\}\mathrm{\ and\ }\mathscr G^2_4=\mathscr E^2_4/\mathrm{span}\{(1,x_2^3),(0,x_2),(0,1)\}.$$ It is easy to check that $\dim(\widetilde K^2_i)_x=2$ for $i=1,4$ and $\dim(\widetilde K^2_i)_x=3$ for $i=2,3$. Hence $x$ is not $\underline r$-strict. Even though $\dim\ker\pi^2_i=3$ for every $1\leq i\leq 4$, the quotients $\{\mathscr E^2_i/\ker\pi^2_i\}_{1\leq i\leq 4}$ do not comprise a linked chain.
\end{nonex}

\begin{thm}\label{thm:dimension of linked determinantal loci}
Suppose $S$ is Noetherian. Let $(\{\mathscr E^j_\bullet\}_{1\leq j\leq m},f_\bullet,f^\bullet)$ be a linked chain of flags, and let $\Delta \sub S$ denote the rank-$\underline r$ linked determinantal locus with respect to the induced morphisms $\pi^j_i$. Assume that $r_j+d_j-r^j_1-r^j_2\geq 0$ for every $j$. Then near any $\underline r$-strict point $x$ in $S$, $\Delta$ has codimension at most 
%Then each irreducible component of the linked determinantal locus associated to $( \{\mathscr E^j_\bullet\}_{1\leq j\leq m},f_\bullet,f^\bullet)$ has codimension at most 
$$\sum_{j=1}^m(r_j-r_{j+1})(r_j+d_j-r^j_1-r^j_2)$$
where by convention we set $r_{m+1}=0$.
\end{thm}
\begin{proof} We may assume $x$ lies in $\Delta$.
Let $S^{\pr}\subset S$ be the subscheme along which the morphism $\pi^m_i$ has rank at most $d_m-r_m$ for every $1\leq i\leq n$. According to \cite[Theorem 1.3]{murray2016linked}, every irreducible component of $S^{\pr}$ has codimension at most $r_m(r_m+d_m-r^m_1-r^m_2)$.

% For $1\leq j\leq m-1$ and $1\leq i\leq n$ let $(\mathscr E^j_i)'=\mathscr E^j_i/\ker\pi^m_i$. We claim that the tuple $( \{(\mathscr E^j_\bullet)'\}_{1\leq j\leq m-1},f_\bullet,f^\bullet)$ gives a 0-linked chain of flags on $S'$.  

%We first check that $(\mathscr E^j_i)'$ is a vector bundle on $S'$. 
\medskip
Recall from Remark \ref{rem:strict linked chain of flags} that, for every $l< i$, we set $f_{l,i}:=f_{i-1}\circ f_{i-2}\circ\cdots\circ f_l$ and $f^{l,i}:=f^l\circ f^{l+1}\circ\cdots\circ f^{i-1}$. We begin by showing that there are distinguished linearly independent vectors $$x_{i,1},\dots,x_{i,k_i}; y_{i,1},\dots,y_{i,l_i}\subset (\mathscr E^m_i)_x$$ for which
\begin{enumerate}
\item[(1)] $\sum_{i=1}^nk_i+l_i=d_m$ and $\sum_{i=1}^n l_i=r_m$;
\item[(2a)] the set of vectors 
$$\{f_{a,i}(X^m_a\cup Y_a),f^{i,b}(X^m_b\cup Y_b))|a<i<b\}\cup X^m_i\cup Y_i$$
generates $(\mathscr E^m_i)_x$, where $X^m_i:=\{x_{i,1},...,x_{i,k_i}\},Y_i:=\{y_{i,1},...,y_{i,l_i}\}$; and
\item[(2b)] the set
$$\{f_{a,i}(Y_a),f^{i,b}( Y_b))|a<i<b\}\cup Y_i$$
generates $\ker(\pi^m_i|_x)$.
\end{enumerate}

\noindent Indeed, the first equations of condition (1) and (2a) are simultaneously satisfied whenever 
\[
\mathrm{span}(X^m_i)\oplus\mathrm{span}(Y_i)\oplus (K^m_i)_x=(\mathscr E^m_i)_x
\]
in which case we set $l_i=\dim \ker(\pi^m_i|_x)-\dim (\widetilde K^m_i)_x$, and choose any $l_i$ independent vectors $y_{i,\lambda}\in \ker(\pi^m_i|_x)\backslash \widetilde K^m_i$, $1\leq \lambda\leq l_i$. 
The set of vectors in (2b) then spans a subspace of $\ker(\pi^m_i|_x)$ of dimension $\sum_{i=1}^nl_i$, which is at least $r_m$ by $\underline r$-strictness. According to Remark~\ref{rem:strict linked chain of flags}, we have $
\dim \ker(\pi^m_i|_x)=r_m=\sum_{i=1}^nl_i$, from which the generation statement in (2b) follows. 

\vspace{10pt} For each $1\leq j \leq m-1$ we now extend $X^m_i$ inductively to $X^j_i$, say $X^j_i=X^{j+1}_i\cup Z^j_i$, such that $\mathrm{span}(X^j_i)\oplus \mathrm{span}(Y_i)\oplus (K^j_i)_x=(\mathscr E^j_i)_x$. This is possible because $x$ is $\underline r$-strict. Hence (2a) remains true when we replace $m$ by $j$. Further, we have $\sum_{i=1}^n|X^j_i|=d_j-r_m$. Let $d_j^{\pr}:=d_j-r_m$, $1\leq j\leq m-1$.

\vspace{10pt}
Let $U$ be an open neighborhood $U$ of $x$ in $S^{\pr}$ over which (the restriction of) each $
\mathscr E^j_i$ becomes free, the kernel of $\pi^m_i$ at each fiber has dimension exactly $r_m$, and each $X^j_i$ lifts to a section in $\mathscr E^j_i(U)$. Fix such a lift and let $\mathcal X_j$ denote the set of lifts of all $X^j_i$. Let $(\mathscr E^j_i)^{\pr}$ be the subsheaf of $\mathscr E^j_i$ over $U$ generated by the image of $\mathcal X_j$. We claim that, after shrinking $U$, there is an induced linked-chain-of-flags structure on $\{(\mathscr E^j_i)^{\pr}\}$ which serves as the ``quotient" of $\{\mathscr E^j_i\}$ by the kernel of $\pi^m_i$. 

\vspace{10pt}
To this end, we first restrict $U$ to the complement of the closed subscheme along which the maps 
$$\mathcal O_U^{\oplus d_j^{\pr}}\rightarrow \mathscr E^j_i$$
induced by the images of $\mathcal X_j$ in $\mathscr E^j_i(U)$
fail to have full rank. By the same argument as that used in \cite[Lemma 2.5]{osserman2011linked}, $(\mathscr E^j_i)^{\pr}$ is a free module over $U$ of rank $d_j^{\pr}$ and there is an inclusion of fibers $(\mathscr E^j_i)^{\pr}_y\hookrightarrow (\mathscr E^{j}_i)_y$ for every $y\in U$. It is straightforward to check that $(\{(\mathscr E^j_\bullet)^{\pr}\}_{1\leq j\leq m-1},f_\bullet,f^\bullet)$ is a linked chain of flags. 

Next shrink $U$ so that the map $\pi^m_i\colon (\mathscr E^m_i)^{\pr}\rightarrow \mathscr G^m_1\oplus\mathscr G^m_n$ has full rank on $U$. Further, assume the sub-bundle of $\mathscr E^j_i$ generated by the images of $\mathcal X_j\backslash \mathcal X_m$ has trivial intersection with $\mathscr E^m_i$, i.e. that the induced map from the sub-bundle to $
\mathscr E^j_i/\mathscr E^m_i$ has full rank. Then for every $1\leq j\leq m-1$, the map $\pi^j_i$ has kernel of dimension at least $r_j$ on $U$ if and only if $\pi^j_i|_{(\mathscr E^j_i)^{\pr}}$ has kernel of dimension at least $r_j^{\pr}=r_j-r_m$.

\vspace{10pt}
We now check that the linked chain of flags $( \{(\mathscr E^j_\bullet)^{\pr}\}_{1\leq j\leq m-1},f_\bullet,f^\bullet)$ is $\underline r^{\pr}$-strict at $x$. By construction, it is easy to see that $(K^j_i)^{\pr}_x\cap(\mathscr E^l_i)^{\pr}_x=(K^l_i)^{\pr}_x$ for every $l\geq j$. On the other hand, we have $(\widetilde K^j_i)^{\pr}_x\oplus \widetilde Y_i\subset (\widetilde K^j_i)_x$, where $\widetilde Y_i$ is the span of $\{f_{a,i}(Y_a),f^{i,b}( Y_b))|a<i,b>i\}$. Hence $$\sum_{i=1}^n\dim (\widetilde K^j_i)'_x\leq \sum_{i=1}^n(\widetilde K^j_i)_x-(n-1)r_m\leq (n-1)r_j^{\pr}$$
for $1\leq j\leq m-1$ because $x$ is $\underline r$-strict with respect to $(\{(\mathscr E^j_\bullet)\}_{1\leq j\leq m},f_\bullet,f^\bullet)$. 

\vspace{10pt}
By induction, the determinantal locus of $( \{(\mathscr E^j_\bullet)'\}_{1\leq j\leq m-1},f_\bullet,f^\bullet)$ has codimension at most 
$$\sum_{j=1}^{m-1}(r_j'-r_{j+1}')(r_j'+d_j'-(r^j_1)'-(r^j_2)')=\sum_{j=1}^{m-1}(r_j-r_{j+1})(r_j+d_j-r^j_1-r^j_2)$$
in $S^{\pr}$. It therefore has codimension at most 
$$\sum_{j=1}^{m-1}(r_j-r_{j+1})(r_j+d_j-r^j_1-r^j_2)+r_m(r_m+d_m-r^m_1-r^m_2)=\sum_{j=1}^m(r_j-r_{j+1})(r_j+d_j-r^j_1-r^j_2)$$
in $S$.
\end{proof}

\section{A moduli scheme for inclusions of (limit) linear series}\label{sec:moduli}

\subsection{Moduli of inclusions of linear series on a smooth curve}\label{subsec:moduli smooth}
In this section, we briefly review the moduli problem for inclusion of linear series over a smooth curve. The main point is to exhibit the moduli scheme as a determinantal locus. Throughout, $C$ will denote an irreducible, smooth, projective curve over some algebraically closed field $K$. 

Hereafter, we fix positive integers $d_1>d_2$ and $r_1>r_2$ and we denote the $(d_1-d_2)$-th symmetric product of $C$ by $C_{d_1-d_2}$. We let $G^{r_1}_{d_1}(C)$ denote the moduli scheme of linear series $\g^{r_1}_{d_1}$ over $C$. It is well-known that on a general curve, $G^{r_1}_{d_1}(C)$ is smooth of dimension 
\begin{equation}\label{eqn:rho}
\rho(g,r_1,d_1):=g-(r_1+1)(r_1+g-d_1).
\end{equation} 
Now set $d=d_1-d_2$, $r=r_2-r_1+d$, and
\begin{equation}\label{eqn:mu}
\mu(d,r_2,r_1):= d-r(r_1+1-d+r)= (d_1-d_2)-(d_1-d_2-r_1+r_2)(r_2+1).
%\mu(d_1-d_2,r_2,r_1):=(d_1-d_2)-(d_1-d_2-r_1+r_2)(r_2+1).
\end{equation}
%When we map $C$ to an $r$-dimensional projective space via 
Given a linear series $(L,V) \in G^{r_1}_{d_1}(C)$, those $d$-tuples of points on $C$ that span $(d-r-1)$-dimensional subspaces of $\mb{P}V^{\ast}$ describe the sublocus of $C_d$ for which the evaluation map
$V \ra H^0(L/L(-p_1-\dots-p_d))$ has rank $(d-r)$. It follows that $\mu(d,r_2,r_1)$
is the {\it expected} dimension of the space of $(d_1-d_2)$-secant $(r_1-r_2-1)$-planes to the image of a fixed $\g^{r_1}_{d_1}$. With these preliminaries in place, we now rigorously define the functor of points associated to our moduli problem.
\begin{defn}\label{defn:inclusion_ls}
Let $C$ be a smooth, irreducible projective curve. %Let $\cG^{r_1,r_2}_{d_1,d_2}(C):(Sch/K) \to (Sets)$ be the functor such that 
For every $K$-scheme $S$, let $\cG^{r_1,r_2}_{d_1,d_2,C}(S)$ denote the set of equivalence classes of objects of the form $(Y,(L,V_2,V_1))$, where $Y$ is an effective divisor of relative degree $d_1-d_2$ on $C\times S$, $(L,V_1)$ is an $S$-family of linear series $\g^{r_1}_{d_1}$, and $(L(-D),V_2)$ is an $S$-family of series $\g^{r_2}_{d_2}$. Here $(Y,(L,V_2,V_1))$ and $(Y,(L^{\pr},V_2^{\pr},V^{\pr}_1))$ are equivalent whenever there exists a line bundle $F$ on $S$ and a line bundle isomorphism $\phi:L\to L^{\pr}\ot p^*F$ for which $p_*\phi(V_i)=V^{\pr}_i$, for $i=1,2$.
\end{defn} 
We shall define a determinantal scheme $G^{r_1,r_2}_{d_1,d_2}(C)$ that represents $\cG^{r_1,r_2}_{d_1,d_2}(C):(Sch/K) \to (Sets)$. To do so, we begin by reserving the following notation for future reference. %To do so, we first specify a list of notations:

%\begin{nt}\label{it:list1}
\begin{itemize}
\item Let $\mL$ denote a Poincar\'e line bundle over $C\times \Pic^d_1(C)$ and let $\ti{\mL}$ denote its pullback to $C\times C_{d_1-d_2}\times \Pic^{d_1}(C)$. 
\item Fix a reduced effective divisor $Z$ of $C$ of sufficiently large degree and $Z^{\pr}$ and $Z^{\pr \pr}$ denote the respective pullbacks $Z^{\pr}=Z\times \Pic^{d_1}(C)$, $Z^{\pr \pr}=Z\times C_{d_1-d_2}\times\Pic^{d_1}(C)$.
\item Let $E\subset C\times C_{d_1-d_2}\times \Pic^{d_1}(C)$ denote the pull-back of the universal divisor of degree $d_1-d_2$.
\item Let $\fl_{r_1,r_2}(\pi_*\mL(Z^{\pr}))$ denote the flag scheme over $\Pic^{d_1}(C)$ of two-term flags of subbundles of $\pi_*\mL(Z^{\pr})$ and let $\mV_{\bullet}=\mV_1\supset\mV_2$ denote the universal flag over it.
\item Set $X:=(C_{d_1-d_2}\times\Pic^{d_1}(C))\times_{\Pic^{d_1}(C)}\fl_{r_1,r_2}(\pi_*\mL(Z'))$, and let $\ti{\mV}_{\bullet}$ denote the pullback of $\mV_{\bullet}$ to $X$. 
\end{itemize}
%\end{nt}
These objects fit into the following commutative diagram:
\[\begin{tikzcd}
&&X\arrow{dl}[swap]{q_1}\ar[dr,"q_2"]&\\
C\times C_{d_1-d_2}\times\Pic^{d_1}(C)\ar[r,"\ti{\pi}"]\ar[d,"p"]&C_{d_1-d_2}\times\Pic^{d_1}(C)\ar[dr,"u"]&&\fl_{r_1,r_2}(\pi_*\mL(Z'))\arrow{dl}[swap]{s}\\
C\times\Pic^{d_1}(C)\ar[rr,"\pi"]&&\Pic^{d_1}(C)
\end{tikzcd}\]
\begin{defn}\label{defn:smooth_moduli}
With notation as %in \ref{it:list1}, 
above, let $G^{r_1,r_2}_{d_1,d_2}(C)$ denote the subscheme of $X$ defined the intersection of the following two determinantal conditions:
\begin{enumerate}
\item $\ti{\mV}_1\to t^*\pi_*(\mL(Z')|_{Z'})$ is zero, where $t=u\circ q_1=s\circ q_2$; and
\item $\ti{\mV}_2\to q^*_1\ti{\pi}_*\ti{\mL}(Z'')/q^*_1\ti{\pi}_*\ti{\mL}(Z''-E )$ is zero.
\end{enumerate}
\end{defn}  
We briefly justify the existence of the morphism involved in the second condition, as follows. The natural morphism $\mV_2\hookrightarrow\mV_1\to s^*\pi_*\mL(Z')$ pulls back to a morphism $\ti{\mV}_2\to  t^*\pi_*\mL(Z')$. Pulling back the natural map $u^*\pi_*\mL(Z')\to \ti{\pi}_*p^*\mL(Z')=\ti{\pi}_*\ti{\mL}(Z'')$ along $q_1$, we get a map $t^*\pi_*\mL(Z')\to q^*_1\ti{\pi}_*\ti{\mL}(Z'')$. Composing this map with the previous one, we get a map $\ti{\mV}_2\to q^*_1\ti{\pi}_*\ti{\mL}(Z'')$, which induces the map in item (2).
\begin{prop}\label{prop:presentation}
The functor $\cG^{r_1,r_2}_{d_1,d_2}(C):(Sch/K) \to (Sets)$ is represented by $G^{r_1,r_2}_{d_1,d_2}(C)$. 
\end{prop}
\begin{proof}
We need to check that each $S$-valued point of our functor defines an $S$-family $(Y,(L,V_2,V_1))$ of inclusion of linear series as specified in Definition~\ref{defn:inclusion_ls}. To this end, fix a morphism $f:S\to G^{r_1,r_2}_{d_1,d_2}(C)$. For the sake of clarity, we update our previous commutative diagram to
\[\begin{tikzcd}
C\times S\ar[r,"\pi_S"]\ar[d]&S\ar[d]\ar[r,"f"]&G^{r_1.r_2}_{d_1,d_2}(C)\arrow{dl}[swap]{q_1}\ar[d,"q_2"]&\\
C\times C_{d_1-d_2}\times\Pic^{d_1}(C)\ar[r,"\ti{\pi}"]\ar[d,"p"]&C_{d_1-d_2}\times\Pic^{d_1}(C)\ar[d,"u"]&\fl_{r_1,r_2}(\pi_*\mL(Z'))\arrow{dl}[swap]{s}\\
C\times\Pic^{d_1}(C)\ar[r,"\pi"]&\Pic^{d_1}(C)&
\end{tikzcd}\]
Notice that the two rectangles on the left are Cartesian.

By appealing to the universal properties of $C_{d_1-d_2}$, $\Pic^{d_1}(C)$ and $\fl_{r_1,r_2}(\pi_*\mL(Z^{\pr}))$, respectively, and composing $f$ with the relevant morphisms, we get a tuple of the form $(Y,L,V_2,V_1)$, where $Y$ is a relative divisor of degree $d_1-d_2$, $L$ is a line bundle of relative degree $d_1$ over $C\times S$ and $V_2\subset V_1\subset (\pi_*\mL(Z'))_S$. In particular, we have $V_i=(\mV_i)_S$ (where the restriction to $S$ is via pullback along $q_2\circ f$), $L=(\mL)_{C\times S}$ (via pullback along the left vertical arrow), and $Y=(Z^{\pr})_{C\times S}$ (via pullback along $C\times S\to C\times C_{d_1-d_2}\times\Pic^{d_1}(C)\to C\times C_{d_1-d_2}$). 

We then need to check that $(L,V_1)$ and $(L(-Y),V_2)$ are $S$-families of $\g^{r_1}_{d_1}$ and $\g^{r_2}_{d_2}$, respectively. Condition (1) in Definition \ref{defn:smooth_moduli} implies that the natural morphism $V_1\to \pi_{S^*}((\mL(Z^{\pr})|_{Z^{\pr}})_{C\times S})$ is zero; since $V_1=(\mV_1)_S$ is a subbundle of $\pi_{S^*}\mL(Z^{\pr})_{C\times S}$ to begin with (see Remark \ref{rem:subbundle}), it follows that $V_1$ is a sub-bundle of $\pi_{S^*}\mL_{C\times S}$ of rank $r_1+1$ in the sense of Definition \ref{defn:subbundles}, i.e. $(L,V_1)$ is an $S$-family of $\g^{r_1}_{d_1}$. It follows that $V_2$ is a subbundle of $\pi_{S^*}\mL_{C\times S}$, and condition (2) in Definition \ref{defn:smooth_moduli} implies that $V_2\hookrightarrow \pi_{S^*}\mL_{C\times S}$ factors through $\pi_{S^*}\mL_{C\times S}(-Y)$ and hence $(L(-Y),V_2)$ is an $S$-family of $\g^{r_2}_{d_2}$.
\end{proof}

\subsection{Inclusions of limit linear series}
Motivated by the theory on smooth curves, we now give a functorial definition of the main object of study of this paper. Note that in general it is not natural to assume that our objects consist of limit $\g^{r_2}_{d_2}$ embedded in limit $\g^{r_1}_{d_1}$, where $d_2<d_1$. A priori, we also want to take into consideration objects of the form $(\mL\ot\mathcal{I}_Z,\{V'_v\}_v)$ embedded inside $(\mL,\{V_v\}_v)$, where $\cI_Z$ is the ideal sheaf of a non-Cartier effective divisor of a nodal curve. 

We make the following assumption on the (families of) curves we consider in the remainder of the paper:
\begin{nt}\label{nt:curves}
Hereafter, $X/B$ denotes a regular smoothing family in the sense of \cite[Definition 3.9]{osserman2014limit}; in particular, $X$ is projective over $B$ (\cite[\S 2]{lichtenbaum1968curves}). We assume that the special fiber $X_0$, whose dual graph we denote by $\Gamma$, is of compact type and equipped with a fixed enriched structure $$(\mathscr O_v)_{v\in V(\Gamma)}=(\mathscr O_X(Z_v)|_{X_0})_{v\in V(\Gamma)},$$ where $Z_v$ is the component of $X_0$ corresponding to $v\in V(\Gamma)$. We denote by $Z^c_v$ the closure of $X_0\backslash Z_v$, and $Z^\circ_v:=Z_v\backslash Z^c_v$.
\end{nt}
\begin{defn}\label{defn:inc_lls}
Given a quadruple of nonnegative integers $(r_1,r_2,d_1,d_2)$ with $r_1>r_2$ and $d_1>d_2$, an {\it inclusion of limit linear series of type $(r_1,r_2,d_1,d_2)$ over $X_0$} consists of the data 
\[(\mL,Z,\{V_{v,2}\subset V_{v,1}\}_{v\in V(\Gamma)})\] in which
\begin{enumerate}
 \item $\mL$ is a line bundle of total degree $d_1$ over $X_0$;
 \item for all $v\in V(\Gamma)$, $V_{v,2}\subset V_{v,1}\subset \Gamma(X_0,\mL(\mtd_v))$ are subspaces of sections for which $(\mL,(V_{v,1})_{v\in V(\Gamma)})$
  is a limit $\g^{r_1}_{d_1}$ in the sense of Definition~\ref{lls-defn} and $(\mL,(V_{v,2})_{v\in V(\Gamma)})$
  is a limit $\g^{r_2}_{d_1}$; and
 \item $Z$ is a finite sub-scheme of length $d_1-d_2$ such that for every $v\in V(\Gamma)$ the evaluation map  
 \[e_{v,Z}:V_{v,2}\to \Gamma(X_0,\mL(\mtd_v)\ot\mO_Z)
 \]  
 is zero.
 \end{enumerate} 
\end{defn}

\subsection{Moduli of inclusions of limit linear series}\label{subsec:moduli limit} % (This section needs to be revised, possibly replace the relative Hilbert scheme $\mathrm{Hilb}^\circ_d(X/B)$ of all $B$-flat subschemes of $X$ with the open subscheme of $\mathrm{Hilb}^\circ_d(X/B)$ parametrizing subschemes of $X$ avoiding the nodes of $X_0$. Also, more explaination to be added.)
%Let $X/B$ be as in Situation~\ref{si:curves}. %then $X$ is projective over $B$ by \cite[\S 2]{lichtenbaum1968curves}. 
Continuing with the conventions of Notation~\ref{nt:curves},
in this subsection we construct the moduli space over $B$ of inclusion of (limit) linear series on $X$ with base points following an idea of \cite[\S 3]{osserman2014limit}. To this end, let $d_1>d_2>0$, $r_1>r_2>0$, $d=d_1-d_2$ such that $d_1-d_2>r_1-r_2$. We have the following theorem, whose proof is carried out in \S \ref{subsubsec: scheme structure of moduli limit}:

\begin{thm}\label{thm:moduli of inclusion of limit linear series}
Let $X$ and $B$ be as above. There exists a proper moduli space $\mathcal G$ over $B$. Its generic fiber $\mathcal G_{\eta}$ is isomorphic to $G^{r_1,r_2}_{d_1.d_2}(X_\eta)$ and parameterizes all inclusions of linear series of type $(r_1,r_2,d_1,d_2)$ over $X_\eta$, while its special fiber contains an open subset parameterizing all inclusions of limit linear series of type $(r_1,r_2,d_1,d_2)$ over $X_0$.
\end{thm}

\medskip
   %Suppose the special fiber $X_0$ of $X$ is of compact type and has dual graph $\Gamma$. Fix an enriched structure $(\mathscr O_v)_{v\in V(\Gamma)}$.
It is worth mentioning that in principle our construction of $\mathcal G$ below also works if $X_0$ is not of compact type; however, we restrict to the compact type case for simplicity's sake. In what follows, we fix a base multidegree $w_0$ on $\Gamma$ with total degree $d_1$, and let $M=M_{w_0}$ denote the set of multidegrees on $\Gamma$ which can be obtained from $w_0$ by twisting in a finite sequence of vertices. Since $X_0$ is of compact type, $M$ is just the set of all multidegrees on $\Gamma$.  Note that there is a directed graph $G$ whose vertices are identified with $M$ and edges are induced by twisting; see \cite[Notation 2.11]{osserman2014limit} for a more general construction where $X_0$ is not necessarily of compact type. 
 
 Now let $w_v\in M$ be the {\it naive} multidegree concentrated on $Z_v$, i.e. that which has degree $d_1$ on $Z_v$ and is zero elsewhere. Let $\overline M=\overline M_{(w_v)_v}$ denote the union over all adjacent pairs $(v,v^{\pr})\subset V(\Gamma)$ of vertices $w\in V(G)$ that lie between $w_v$ and $w_{v'}$ (so the divisor indexed by $w$ is effective and has degree zero on $V(\Gamma)\backslash \{v,v'\}$). See \cite[Definition 5.5]{osserman2014limit} for a construction of $\overline M_{(w_v)_v}$ for more general concentrated degrees $(w_v)_v$. We will use the terminology {\it limit $\mathfrak g^r_{w_0}$} to mean a limit linear series of rank $r$ and multidegree $w_0$ on $X_0$.

Finally, let $\mathrm{Hilb}_d(X/B)$ denote the ``relative" Hilbert scheme parameterizing closed finite degree $d$ subschemes of the fibers of $X$ over $B$.
%, flat over $B$, with relative Hilbert polynomial
%\footnote{are those  subschemes set-theoretically a union of closures of points of $X_\eta$?} 
%; see for example \cite{osserman256pithy}. 
Let $\mathrm{Hilb}^\circ_d(X/B)\subset \mathrm{Hilb}_d(X/B)$ be the open subscheme parameterizing subschemes of $X$ supported away from the nodes of $X_0$. Given $w\in M$, let $\mathrm{Pic}^w(X/B)$ denote the moduli schemes of line bundles of degree $d_1$ that have multidegree $w$ on $X_0$. 

\subsubsection{The schematic structure of the moduli space}\label{subsubsec: scheme structure of moduli limit} The moduli space of limit linear series constructed in \cite{osserman2014limit} is isomorphic as a scheme to the Eisenbud--Harris limit linear series moduli space for compact-type curves over the locus of refined series. As we build on Osserman's construction, we will freely use the notions of ``subbundles of pushforwards" and ``generalized determinantal loci" introduced in \cite[Appendix B]{osserman2014higherlimit}. 

\medskip
To begin, let %$G^{r}_w(X/B)$ be the moduli space of linear series on $X$ with multidegree $w$ on $X_0$.
%Let 
$G^{r_1}_{w_0}(X/B)$ (resp., $G^{r_1}_{w_0}(X_0)$) denote the moduli space of (limit) linear series on $X$ with multidegree $w_0$ on $X_0$ (resp., the moduli space of limit $\mathfrak g^{r_1}_{w_0}$s on $X_0$) with respect to the fixed enriched structure $(\mathscr O_v)_v$ as in Notation \ref{nt:curves}. These spaces are introduced in \cite[Definition 3.7, Notation 3.13]{osserman2014limit} as closed subschemes of the $\mathrm{Pic}^{w_0}(X/B)$-fiber product of $G^{r_1}_{w_v}(X)$ for all $v$. Here $G^r_w(X)$ is the moduli scheme of pairs $(L,V)$, where $L$ is in $\mathrm{Pic}^w(X/B)$ and $V$ is an $(r+1)$-dimensional space of global sections of $L$. So let $\mathcal P=\mathrm{Hilb}_d(X/B)\times_B G^{r_1}_{w_0}(X)$.
We have the following diagram:
\begin{equation}\label{basic_diagram}
\begin{tikzcd}
&\mathcal P\times_B X\rar\dar&G^{r_1}_{w_0}(X/B)\times_B X\dar{\tilde p}\\\mathcal P\times_B X\rar\dar
&\mathcal P\rar{\tilde q}\dar{q} & G^{r_1}_{w_0}(X/B)&\\
\mathrm{Hilb}_d(X/B)\times_B X\rar{p}& \mathrm{Hilb}_d(X/B). &&
\end{tikzcd}
\end{equation}

For notational simplicity, we abusively omit writing pullbacks whenever doing so is unambiguous. %of the denoted maps above by themselves, if there's no confusion. 
Our construction of course depends on various naturally-arising bundles lying over the schemes in the basic diagram \eqref{basic_diagram}. Accordingly, let $\mathscr L$ denote the universal (Poincar\'e) line bundle over $G^{r_1}_{w_0}(X/B)\times_B X$, and let $\mathscr V^1_v\subset \tilde p_*\mathscr L_{w_v}$ denote the universal subbundle, where $\mathscr L_{w_v}$ is the universal bundle with multidegree $w_v$ on $X_0$ obtained from $\mathscr L$ by twisting with line bundles in the enriched structure. %so that the points in $\mathscr L_{w_v}$ have multidegree $w_v$ on $X_0$.

{\flushleft Let}
\[
\psi_v\colon G_v:=Gr(r_2,\tilde q^*\mathscr V^1_v)\rightarrow \mathcal P
\]
denote the relative Grassmann bundle of $r_2$-dimensional subspaces in the fibers of $\tilde q^*\mathscr V^1_v$, with universal subbundle $\mathscr V^2_v\subset  \psi_v^*\tilde q^*\mathscr V^1_v$. Let $\mathcal G^1$ be the fiber product of all $G_v$ over $\mathcal P$ with projection map $\widetilde\psi_v\colon \mathcal G^1\rightarrow G_v$, and let $\psi:= \psi_v\circ\widetilde \psi_v$.
 
 \medskip
 It is easy to check that $\mathcal G^1_\eta$ parameterizes \textit{generic tuples} $( L(-E)\hookrightarrow  L,(W^2_v)_{v\in V(\Gamma)},W_1)$, in which $E\in\mathrm{Hilb}_d(X_\eta)$, $( L,W_1)\in G^{d_1}_{r_1}(X_\eta)$ and $W^2_v\in Gr(r_2,W_1)$. Similarly $\mathcal G^1_0$ parameterizes \textit{special tuples}  
$( L^{\pr}\otimes  I_{E^{\pr}}\hookrightarrow L^{\pr},(V^2_v,V^1_v)_{v\in V(\Gamma)})$, in which $( L^{\pr},(V^1_v)_v)\in G^{r_1}_{w_0}(X_0)$ is a limit $\mathfrak g^{r_1}_{w_0}$ on $X_0$, $ I_{E^{\pr}}$ is the ideal sheaf of $E^{\pr}\in\mathrm{Hilb}_d(X_0)$ and $V^2_v\subset Gr(r_2,V^1_v)$. 

\medskip
Let $\mathcal G^2\subset\mathcal G^1$ denote the intersection of the (generalized) $(r_2+1)$-st vanishing loci of the maps
\begin{equation}\label{eqn:G2}\tilde p_* \psi^*\tilde q^*\mathscr L_w\rightarrow \displaystyle\bigoplus_{v\in V(\Gamma)}\tilde p_* \psi^*\tilde q^*\mathscr L_{w_v}/\widetilde \psi^*_v\mathscr V^2_v\end{equation}
over all $w\in M$. Then $\mathcal G^2_\eta$ corresponds to generic tuples with $W^2_v$ independent of $v$, and $\mathcal G^2_0$ parameterizes special tuples satisfying $( L^{\pr},(V^2_v)_v)\in G^{r_2}_{w_0}(X_0)$. 

\medskip
Let $\mathcal I$ be the universal ideal sheaf over $ \mathrm{Hilb}_d(X/B)\times_B X$. Then $\psi^*q^*\mathcal I$, as a coherent sheaf on $\mathcal G^1\times_B X$, is flat over $\mathcal G^1$. Note that $\widetilde \psi^*_v\mathscr V^2_v$ is a subbundle of $\tilde p_*\psi^*\tilde q^*\mathscr L_{w_v}$, and we have an exact sequence
$$0\rightarrow\psi^*\tilde q^*\mathscr L_{w_v}\otimes\psi^*q^*\mathcal I\rightarrow  \psi^*\tilde q^*\mathscr L_{w_v}\rightarrow \mathscr L_{w_v}|_\mathcal I:=\psi^*\tilde q^*\mathscr L_{w_v}/(\psi^*\tilde q^*\mathscr L_{w_v}\otimes\psi^*q^*\mathcal I)\rightarrow 0$$ on $\mathcal G^1\times_B X$ that is preserved under the base change $T\rightarrow \mathcal G^1$. Furthermore, the pushforward $\tilde p_*\mathscr L_{w_v}|_\mathcal I$ is a vector bundle of rank $d$ over (the preimage in $\mathcal G^1$ of) $\mathrm{Hilb}^\circ_d(X/B)$ by \cite[\S 0.5]{mumford1994geometric}. Our moduli space over $B$ will be the locus $\mathcal G$ in $\mathcal G^2$ along which the composition 
\begin{equation}\label{eqn:G}\widetilde \psi^*_v\mathscr V^2_v\hookrightarrow\tilde p_* \psi^*\tilde q^*\mathscr L_{w_v}\rightarrow \tilde p_*\mathscr L_{w_v}|_\mathcal I\end{equation} vanishes identically.

%Let $\mathcal G^2\subset \mathcal G^1$ be the (generalized) $(r_2+1)$-st vanishing locus of the composed map (note that $p=\tilde p$ on $\mathcal G^1$)
%$$\tilde p_*(\psi^*\tilde q^*\mathscr L_{w_v}\otimes\psi^*q^*\mathcal I)\rightarrow \tilde p_* \psi^*\tilde q^*\mathscr L_{w_v}\rightarrow \tilde p_* \psi^*\tilde q^*\mathscr L_{w_v}/\widetilde \psi^*_v\mathscr V^2_v.$$
%Then plainly $\mathcal G^2_\eta$ parametrizes generic tuples such that $W^2_v\subset W_1(-E))$ and $\mathcal G^2_0$ parametrizes special tuples such that $V^2_v\subset V^1_v(-E)$.

%Our moduli space over $B$ will be the locus $\mathcal G$ in $\mathcal G^2$ such that $\mathcal G_\eta$ corresponds to generic tuples with $W^2_v$ independent of $v$ and $\mathcal G_0$ parametrizes special tuples satisfying $(\mathcal L',(V^2_v)_v)\in G^{r_2}_{w_0}(X_0)$. More precisely, $\mathcal G$ is the intersection of the $(r_2+1)$-st vanishing locus of the map
%$$\tilde p_* \psi^*\tilde q^*\mathscr L_w\rightarrow \displaystyle\bigoplus_{v\in V(\Gamma)}\tilde p_* \psi^*\tilde q^*\mathscr L_{w_v}/\widetilde \psi^*_v\mathscr V^2_v$$
%over all $w\in M_{w_0}$.

\medskip
By construction, a $T$-valued point of $\mathcal G^1$ is a tuple $(L_T, E_T,( V^2_v)_v,( V^1_v)_v))$ where 
\begin{itemize}
\item $L_T$ is a line bundle over $T\times_BX$ with multidegree $w_0$ on $T\times_B X_0$; 
\item $ E_T$ is a closed subscheme of $T\times_BX$, flat and with Hilbert polynomial $d$ over $T$;
\item $ V^1_v$ is a rank-$(r_1+1)$ subbundle of $\tilde p_* L_{T,w_v}$ for which the map
$$\tilde p_* L_{T,w}\rightarrow \displaystyle\bigoplus_{v\in V(\Gamma)}\tilde p_*  L_{T,w_v}/ V^1_v$$ has $(r_1+1)$-st vanishing locus equal to $T$ for all $w\in M$; and 
\item $ V^2_v$ is a rank-$(r_2+1)$ subbundle of $ V^1_v$. 
\end{itemize}
Here $ L_{T,w_v}$ (and similarly for $ L_{T,w}$) is defined in the same manner as $\mathscr L_{w_v}$. Such a tuple is a $T$-valued point of $\mathcal G$ whenever the map
$$\tilde p_* L_{T,w}\rightarrow \displaystyle\bigoplus_{v\in V(\Gamma)}\tilde p_*  L_{T,w_v}/ V^2_v$$ has $(r_2+1)$-st vanishing locus equal to $T$ for all $w\in M$, and $$ V^2_v\rightarrow \tilde p_*( L_{T,w_v}/ L_{T,w_v}\otimes  I_{ E}))$$ vanishes identically on $T$ for all $v\in V(\Gamma)$, where $ I_{ E}$ is the ideal sheaf of $ E$.
\begin{rem}\label{rem: recording base points}
For technical convenience, our moduli functor keeps track of the base-locus of the sub-series as well as the sub-series itself. One can certainly recover the moduli problem of (limit) linear series admitting sub-series with prescribed amount of base points by applying the obvious forgetful functor.
\end{rem}

Finally, let $\mathcal G^\circ$ denote the preimage of $\mathrm{Hilb}^\circ_d(X/B)$ in $\mathcal G$. Then $\mathcal G^\circ_0\subset \mathcal G_0$ is the open subset described in Theorem \ref{thm:moduli of inclusion of limit linear series}.

%, and $G^{r_1}_{w}(X/B)$ the moduli scheme of pairs $(\mathscr L,V)$ where $\mathscr L$ is in $\mathrm{Pic}^w(X/B)$ and $V$ is an $(r_1+1)$-dimensional space of global sections of $\mathscr L$ (\cite[\S 5]{osserman2006limit}).

\subsubsection{An alternative construction.}\label{subsubsec: alternate construction}
In this subsection we give another construction of a moduli space over $B$ of inclusion of (limit) linear series, which at least set-theoretically agrees with $\mathcal G^\circ$. We construct the alternative space as an intersection of (linked) determinantal loci of vector bundles over a regular ambient space. The alternative construction makes estimating the dimension of $\mathcal G^\circ$ more transparent and is also helpful in proving the smoothing of certain well-behaved inclusions of limit linear series; see \S \ref{subsec:smoothing}. %Note also that we work over $\mathrm{Hilb}^\circ_d(X/B)$ for simplicity.

\medskip
To begin, let $\mathscr L_w$ denote the universal bundle over $\mathrm{Pic}^w(X/B)\times_B X$.
  Fix a choice of effective divisor $D=\sum_{v\in V(\Gamma)}D_v$ on $X$ of relative degree $\tilde d=\sum \deg D_v$ in which $D_v$ is a union of sections of $X/B$ passing through $Z^\circ_v$ and $d_v=\deg D_v$ is sufficiently large relative to $d$ and the genus $g_v$ of $Z_v$. %$g_v=g(Z_v)$.
  In fact, we will see that it suffices to choose $D_v$ with $d_v- d>2g_v-2$. 
%Let $\mathrm{Hilb}^\circ_d(X/B,D)$ be the subspace of $\mathrm{Hilb}^\circ_d(X/B)$ parametrizing subschemes of $X$ disjoint with $D_0:=D\cap X_0$. 

 Now set $$\widehat{\mathcal P}:= \mathrm{Hilb}^\circ_d(X/B)\times_B\prod_{v\in V(\Gamma)} \mathrm{Pic}^{w_v}(X/B)$$ where the product of Picard varieties is over $\mathrm{Pic}^{w_0}(X/B)$.
The following commutative diagram will be useful in tracking (the various projections involved in) our construction.
$$
\begin{tikzcd}
&\widehat{\mathcal P}\times_B X\rar\dar&\mathrm{Pic}^{w_v}(X/B)\times_B X\dar{p_v}\rar{\pi_v}&X\\\widehat{\mathcal P}\times_B X\rar\dar
&\widehat{\mathcal P}\rar{q_v}\dar{\hat q} & \mathrm{Pic}^{w_v}(X/B)&\\
\mathrm{Hilb}^\circ_d(X/B)\times_B X\rar{p}& \mathrm{Hilb}^\circ_d(X/B). &&
\end{tikzcd}
$$
 
 Let $\mathcal L_v:=q_v^*(\mathscr L_{w_v}\otimes\pi_v^*\mathcal O_X(D))$ and $\mathcal E_v=p_{v*}\mathcal L_v$. Similarly, let $\mathcal L_v^{\pr}:=q_v^*(\mathscr L_{w_v}\otimes \pi_{v}^*\mathcal O_X(D)|_{D_v})$ and $\mathcal E^{\pr}_v=:p_{v*}\mathcal L_v'$ over $\widehat{\mathcal P}$. By virtue of our choice of $d_v$, $\mathcal E_v$ (resp., $\mathcal E^{\pr}_v$) is a rank-$(d_1+\tilde d-g+1)$ (resp., rank-$ d_v$) vector bundle. 
 
 Let $F_v:=  \mathrm{Flag}(r_2+1,r_1+1;\mathcal L_v)$ denote the relative flag variety over $\widehat{\mathcal P}$, and let $\mathcal I^\circ$ denote the universal ideal sheaf on $\mathrm{Hilb}^\circ_d(X/B)\times_BX$. Let $\widetilde {\mathcal E}_v:=p_{v*}(\mathcal L_v\otimes \hat q^*\mathcal I^\circ)$ and $\widetilde {\mathcal E}^{\pr}_v=p_{v*}(\mathcal L_v^{\pr}\otimes \hat q^*\mathcal I^\circ)$; note that $p=p_v$ on $\widehat{\mathcal P}\times_B X$ for each $v$. We have $\widetilde{\mathcal E}_v\hookrightarrow \mathcal E_v$ and it follows from our choice of $d_v$ and \cite[\S 0.5]{mumford1994geometric} that $\widetilde{\mathcal E}_v$ (resp., $
\widetilde{\mathcal E}_v^{\pr}$) is a vector bundle
%\footnote{need to check: if $E\in\mathrm{Hilb}^\circ_d(X_0)$ and $\mathscr L$ a line bundle on $X_0$ of degree $d_1+\tilde d$, then $h^0(X_0, \mathscr L\otimes \mathscr I_E)=d_2+\tilde d-g+1,$ or $h^1(X_0,\mathscr L\otimes \mathscr I_E)=0.$}
of rank $d_2+\tilde d-g+1$ (resp., $
\tilde d_v$).

\medskip
\textbf{The total space.} The generic fiber of the $\widehat{\mathcal P}$-fiber product $$\mathcal G^1_D:=\prod_{v\in V(\Gamma)}F_v$$ over $B$ parameterizes \textit{generic $D$-tuples} $( L(-E)\hookrightarrow \mathscr L,(W^2_v,W^1_v)_{v\in V(\Gamma)})$ in which $E\in\mathrm{Hilb}_d(X_\eta)$, $ L\in\mathrm{Pic}^{d_1}(X_\eta)$ and $$(W^2_v,W^1_v)\in \mathrm{Flag}(r_2+2,r_1+1; H^0(X_\eta, L(D_\eta))).$$ Similarly the special fiber of $\mathcal G^1_D$ parameterizes \textit{special $D$-tuples}  
$( L^{\pr}(-E^{\pr})\hookrightarrow L^{\pr},(V^2_v,V^1_v)_{v\in V(\Gamma)})$, in which $ L^{\pr}\in\mathrm{Pic}^{w_0}(X_0)$, $E^{\pr}\in\mathrm{Hilb}^\circ_d(X_0)$ and $$(V^2_v,V^1_v)\in \mathrm{Flag}(r_2+1,r_1+1; H^0(X_0, L^{\pr}_{w_v}(D_0))).$$
Here $ L^{\pr}_{w_v}$ is the line bundle with multidegree $w_v$ obtained from  $ L^{\pr}$ by twisting. 

\begin{rem}\label{rem: total space}
We will realize the alternative moduli space as a closed subscheme of $\mathcal G^1_D$, rather than $\mathcal G^1$. We do so for two reasons. First, it enables us to describe the moduli space as a determinantal locus associated with several linked chains of flags (see the construction below), which in turn enables us to leverage the dimension estimate of Section~\ref{sec:linked} to obtain a lower bound for its dimension. On the other hand, to prove the universal openness of the moduli space in Theorem~\ref{thm:smoothing theorem} we need an ambient space that is smooth %\footnote{check that $\mathrm{Hilb}^\circ_d(X/B)$ is smooth over $B$}
over $B$; see also \cite[Corollary 5.1]{osserman2015relative}.
\end{rem}

%\textbf{Condition imposed by $W^2_v\subset W^1_v$ and $V^2_v\subset V^1_v$.}
{\flushleft Now} consider the projection 
$$\mathcal G^1_D\xrightarrow{\widetilde \phi_v}F_v\xrightarrow{\phi_v}\widehat{\mathcal P}.$$
Let $\phi:= \phi_v\circ\widetilde \phi_v$, and let $\mathscr V^2_v\hookrightarrow \mathscr V^1_v\hookrightarrow\phi_v^*\mathcal E_v$ denote the universal flag over $F_v$ whose pullback to $\mathcal G^1_D$ we denote by $\mathcal V^2_v\hookrightarrow \mathcal V^1_v\hookrightarrow\phi^*\mathcal E_v$. %Let  $\mathcal G^2_D\subset \mathcal G^1_D$ be the locus over which the map $\mathcal V_{2,j}\rightarrow \pi^*_{2,j}\widetilde\pi^*_{2,j}\mathcal L_j/\mathcal V_{1,j}$ vanishes identically.

\medskip
\textbf{Conditions imposed by $W^1_v\subset H^0(X_\eta, L(D_\eta-(D_v)_\eta))$ and $V_v^1\subset H^0(X_0,  L^{\pr}_{w_v}(D_0-(D_v)_0)$).}
 The locus $\mathcal G^2_D\subset \mathcal G^1_D$ of points over which the map $\mathcal V^1_v\rightarrow \phi^*\mathcal E'_v$ vanishes identically for each $v\in V(\Gamma)$ has generic fiber parameterizing generic $D$-tuples such that $W^1_v\subset H^0(X_\eta, L(D_\eta-(D_v)_\eta))$ and special fiber parameterizing special $D$-tuples such that $V_v^1\subset H^0(X_0,  L_{w_v}(D_0-(D_v)_0))$ for each $v\in V(\Gamma)$.

\medskip
\textbf{Conditions imposed by $W^2_v\subset H^0(X_\eta,  L(-E))$ and $V^2_v\subset H^0(X_0, L'_{w_v}(-E^{\pr}))$.}
%Let $\mathcal O(D)=q_v^* \pi_v^*\mathcal O_X(D)$. We consider the following exact sequence on $\widehat{\mathcal P}\times_B X$
On $\widehat{\mathcal P}\times_B X$, we have an exact sequence
$$0\rightarrow \hat q^*\mathcal I^\circ\rightarrow \mathcal O(D)\rightarrow \mathcal O(D)/\hat q^*\mathcal I^\circ\rightarrow 0$$
where $\mathcal O(D)=q_v^* \pi_v^*\mathcal O_X(D)$. 
Tensoring with $q_v^*\mathscr L_{w_v}$ we get
$$0\rightarrow  q_v^*\mathscr L_{w_v}\otimes \hat q^*\mathcal I^\circ\rightarrow \mathcal L_v= q_v^*\mathscr L_{w_v}\otimes \mathcal O(D)\rightarrow   q_v^*\mathscr L_{w_v}\otimes(\mathcal O(D)/\hat q^*\mathcal I^\circ)\rightarrow 0.$$

Note that $\mathcal O(D)/\hat q^*\mathcal I^\circ$ is flat over $\widehat{\mathcal P}$, as %\footnote{check}
$$0\rightarrow \mathcal O_{\widehat{\mathcal P}\times_B X}/\hat q^*\mathcal I^\circ\rightarrow \mathcal O(D)/\hat q^*\mathcal I^\circ\rightarrow \mathcal O(D)/\mathcal O_{\widehat{\mathcal P}\times_B X}\rightarrow 0$$
is exact. It follows that $p_{v^*}( q_v^*\mathscr L_{w_v}\otimes(\mathcal O(D)/\hat q^*\mathcal I^\circ))$ is a vector bundle on $\widehat{\mathcal P}$ of rank $\tilde d+d$. Let $\mathcal G^3_D\subset \mathcal G^2_D$ denote the locus over which the composition
$$\mathcal V^2_v\rightarrow\phi^* \mathcal E_v=\phi^*p_{v*}( q_v^*\mathscr L_{w_v}\otimes \mathcal O(D))\rightarrow \phi^*p_{v*}( q_v^*\mathscr L_{w_v}\otimes(\mathcal O(D)/q^*\mathcal I^\circ))$$ vanishes identically.

Similarly, by applying \cite[\S 0.5]{mumford1994geometric} %\mathcal L_v/ \widetilde {\mathcal L}_v=
we see that $p_{v*}(\mathcal L_v/\mathcal L_v\otimes \hat q^*\mathcal I^\circ)$ is an vector bundle %\footnote{again, need $h^1(X_0,\mathscr L\otimes \mathscr I_E)=0$ for $
%\mathscr L$ of degree $d_1+\tilde d$.} 
over $\widehat{\mathcal P}$ of rank $d$ whose fiber at a $B_\eta$-point (resp., $B_0$-point) is $H^0(X_\eta,  L(D_\eta))/H^0(X_\eta,  L(D_\eta-E))$   
(resp., $H^0(X_0, L'_{w_v}(D_0))/H^0(X_0, L'_{w_v}(D_0-E'))$) for some $ L$ and $E$ (resp., $ L^{\pr}$ and $E^{\pr}$).
Let $\widetilde{\mathcal G}^3_D\subset \mathcal G^2_D$ denote the locus along which the map $\mathcal V^2_v\rightarrow  \phi^*p_{v*}(\mathcal L_v/\mathcal L_v\otimes \hat q^*\mathcal I^\circ)$ vanishes identically. Then $\widetilde{\mathcal G}^3_D$ parameterizes generic $D$-tuples such that $W^2_v\subset H^0(X_\eta,  L(D_\eta-E))$ and special $D$-tuples such that $V^2_v\subset H^0(X_0, L_{w_v}^{\pr}(D_0-E^{\pr}))$. The space $\widetilde{\mathcal G}^3_D$ will rejoin our discussion at the end of this subsection.

\medskip
\textbf{Conditions imposed by the linked chain of flags.}
By construction, $\mathcal V^2_v$ is a subbundle of $\phi^*\widetilde{\mathcal E}_v=p_{v*}\phi^*(\mathcal L_v\otimes \hat q^*\mathcal I^\circ)$ of rank $r_2+1$ over $\mathcal G^3_D$. % by looking at the fibers. 
 Hence $\phi^*\widetilde{\mathcal E}_v/\mathcal V^2_v$ is a vector bundle of rank $d_2+\tilde d-g-r_2$ over $\mathcal G^3_D$.
%$\phi^*\widetilde{\mathcal L}_v/\mathcal V^2_v$ is a vector bundle of rank $d_2+\tilde d-g-r_2$ over $\mathcal G^3_D$ by looking at the fibers.
Note also that for each $w\in M$ we have a commutative diagram
$$
\begin{tikzcd}
\widehat{\mathcal P}\times_B X\rar\dar&\mathrm{Pic}^{w_v}(X/B)\times_B X\dar{p_v}\rar&\mathrm{Pic}^{w}(X/B)\times_B X\dar\rar{\pi_w} &X\\
\widehat{\mathcal P}\rar{q_v} & \mathrm{Pic}^{w_v}(X/B)\rar{q^v_w}&\mathrm{Pic}^{w}(X/B)&
\end{tikzcd}
$$
in which $q^v_w$ is defined by twisting. Let $\mathcal G^4_D\subset \mathcal G^3_D$ denote the locus along which the morphism 
\begin{equation}\label{G4D}
\phi^*p_{v^*}q_v^*q_w^{v^*}(\mathscr L_w\otimes \pi_w^*\mathcal O_X(D))\rightarrow \phi^*\mathcal E_v/\mathcal V^1_v\oplus \phi^*\mathcal E_{v^{\pr}}/\mathcal V^1_{v^{\pr}}
\end{equation}
induced by twisting with the enriched structure (see also \cite[Remark 2.22]{osserman2014limit})
has rank at most $d_1+\tilde d-g-r_1$ for every $w\in\overline M$ between $w_v$ and $w_{v^{\pr}}$, where $(v,v^{\pr})$ varies over all pairs of adjacent vertices in $\Gamma$. This determinantal condition was used by Osserman to construct the moduli space of limit linear series \cite[\S 3]{osserman2014limit}. Note also that the domain vector bundle in \eqref{G4D}
is independent of $v$, as $q^v_w$ factors through
$$\mathrm{Pic}^{w_v}(X/B)\rightarrow \mathrm{Pic}^{w_0}(X/B)\rightarrow \mathrm{Pic}^{w}(X/B).$$

Finally, let $\mathcal G_D\subset \mathcal G^4_D$ denote the locus along which the morphism 
\begin{equation}\label{eqn:G2D}\phi^*p_{v^*}(q_v^*q_w^{v^*}(\mathscr L_w\otimes \pi_w^*\mathcal O_X(D))\otimes q^*\mathcal I^\circ)\rightarrow \phi^*\widetilde{\mathcal E}_v/\mathcal V^2_v\oplus\phi^* \widetilde{\mathcal E}_{v^{\pr}}/\mathcal V^2_{v^{\pr}}\end{equation}
has rank at most $d_2+\tilde d-g-r_2$ for any $w\in \overline M$ between $w_v$ and $w_{v^{\pr}}$, where $v$ and $v^{\pr}$ are as above. Arguing as in \cite[Theorem 6.1]{osserman2014limit}, we see that the generic tuples (resp., special tuples) in $\mathcal G^4_D$ are such that $W^1_v\subset H^0(X_\eta, L)$ and $W^1_v$ is independent of $v$ (resp., $V^1_v\subset H^0(X_0, L^{\pr}_{w_v})$ and $( L^{\pr},(V^1_v)_v)$ is a limit $\mathfrak g^{r_1}_{w_0}$ on $X_0$ of multidegree $w_0$).  In addition, the generic tuples (resp., special tuples) in $\mathcal G_D$ are such that $W^2_v\subset H^0(X_\eta, L(-E))$ and $W^2_v$ is independent of $v$ (resp. (i) $V^2_v\subset H^0(X_0, L^{\pr}_{w_v}(-E^{\pr}))$ and (ii) $( L^{\pr}(-E^{\pr}),(V^2_v)_v)$ is a limit linear series on $X_0$, where, given (i), (ii) is equivalent to that $( L^{\pr},(V^2_v)_v)$ is a limit $\mathfrak g^{r_2}_{w_0}$ on $X_0$). Summing up, we have proved the following proposition:

%Then $\mathcal G_D$ is the desired space over $R$ whose generic fiber is the moduli space of included linear series and special fiber parametrizes inclusion of limit linear series.

%It follows by construction and argueing as in \cite[Proposition 6.4]{osserman2014limit} that we have the following identity of moduli spaces.

\begin{prop}\label{prop:two constructions compatible}
 The constructions of Subsections~\ref{subsubsec: scheme structure of moduli limit} and \ref{subsubsec: alternate construction} agree set-theoretically, i.e. we have $|\mathcal G^\circ|=|\mathcal G_D|$.
\end{prop}

\begin{rem}\label{rem:scheme structure of moduli}
It is unclear whether $\mathcal G^\circ $ and $\mathcal G_D$ are isomorphic as {\it schemes}. The main subtlety is that the determinantal condition in \eqref{eqn:G2} for $\mathcal G$ is imposed for all $w\in M$, whereas in \eqref{eqn:G2D} it is only imposed for $w\in \overline M$. This is already significant at the level of the moduli space of limit linear series; see, for example, the proof of \cite[Proposition 3.2.7]{Lieblich2018}.
\end{rem}

\begin{rem}\label{rem:G3D}
Let $\widetilde{\mathcal G}_D$ denote the subscheme of $\widetilde{\mathcal G}^3_D$ cut out by the same determinantal condition as $\mathcal G_D$ in $\mathcal G^3_D$ (here note that $\mathcal V^2_v$ is also a subbundle of $\phi^*(\widetilde{\mathcal E}_v)$ over $\widetilde{\mathcal G}^3_D$). Let $$\mathrm{Hilb}^{\square}_d (X/B)\subset\mathrm{Hilb}^\circ_d (X/B) $$ denote the open subscheme parameterizing subschemes of $X$ that avoid both $D_0$ and the nodes of $X_0$. Then the preimage of $\mathrm{Hilb}^\square_d(X/B)$ in $\mathcal G_D$ agrees set-theoretically with its preimage in  $\widetilde{\mathcal G}_D$, both of which parameterize inclusions of linear series (resp., limit linear series) with base points disjoint with $D_\eta$ (resp., $D_0$). Indeed, for a generic $D$-tuple and $E\in \mathrm{Hilb}^\square_d (X/B)$, if $W^2_v\subset H^0(X_\eta, L(D_\eta-E_\eta))$ and $W^1_v\subset H^0(X_\eta,  L)$, then $W^2_v\subset H^0(X_\eta, L(-E_\eta))$, and similarly for the special $D$-tuples.

%parametrizes inclusions of (limit) linear series where the base points are away from the nodes of $X_0$ and $D$. Indeed, for a generic $D$-tuple and $E\in \mathrm{Hilb}^\circ_d (X/B)$, if $V_2\subset H^0(X_\eta,\mathscr L(D_\eta-E_\eta))$ and $V_2\subset H^0(X_\eta, \mathscr L)$, then $V_2\subset H^0(X_\eta,\mathscr L(-E_\eta))$, and similarly for the special $D$-tuples.
\end{rem}

%\begin{rem}
 %Ideally $\mathcal G_D\subset\mathcal G^5_D$ %be the locus where $V_2\subset H^0(X_\eta,\mathscr L(-E))$, this 
% may be obtained by looking at the determinantal locus of the composed map $\mathcal V_{2,j}\rightarrow \widetilde L_j\rightarrow \widetilde L_j'$. However, it is unclear whether the exact sequence
%$$ q_j^*\widetilde{\mathscr L}_{w_j}\otimes q^*_0\mathcal I^\circ\rightarrow q_j^*\widetilde{\mathscr L}_{w_j}(\pi_j^*D)\otimes q^*_0\mathcal I^\circ\rightarrow (q_j^*\widetilde{\mathscr L}_{w_j}(\pi_j^*D)\otimes q^*_0\mathcal I^\circ)|_D\rightarrow 0 $$
%is injective on the left. Hence the kernel of the fiber of the map $\widetilde{\mathcal L}_j\rightarrow \widetilde{\mathcal L}'_j$ is not necessarily identified with each $H^0(X_\eta,\mathscr L(-E))$.
%\end{rem}

\subsection{A smoothing theorem}\label{subsec:smoothing}
In this subsection we assume that 
%$\Gamma$ is a path graph, and 
$X_0$ is a (Brill--Noether) general chain of $g$ smooth elliptic curves. We also assume that $\rho(g,r_1,d_1)=\mu(d_1-d_2,r_2,r_1)=0$. The facts that $\rho=0$ and that $X_0$ is general imply that the Brill--Noether variety of limit linear series $G^{r_1}_{d_1}(X_0)$ is a reduced zero-dimensional scheme, whose cardinality is given by the number of $(g-d_1+r_1) \times (r_1+1)$-dimensional Young tableaux with entries in $[g]=\{1,\dots,g\}$; see, e.g., \cite{LT}. An equivalent indexing scheme for the elements of $G^{r_1}_{d_1}(X_0)$, more efficient for our purposes, is by length-$g$ strings of numbers $i_1 i_2 \dots i_g$, in which each index $i_j$ belongs to $[r_1+1]$. (Indeed, given a tableau $(T_{ij})$, the $T_{ij}$th letter of the corresponding word is $j$; conversely, given a word $w$, we set $T_{ij}$ equal to the $i$th instance of the number $j$ in $w$.) Hereafter we refer to the string that indexes a limit linear series as its {\it word type}. %More precisely, each element of $G^{r_1}_{d_1}(X_0)$ is uniquely prescribed by the collection of vanishing sequences of its {\it aspects} along each of the $g$ components, and the column $j=j(i)$ in which a given number $i \in [g]$ appears is the unique index whose vanishing order does not increase in passing from the $i$th to the $(i+1)$st component. In what follows, we will refer to $j(i)$ as the {\it distinguished index} of the $i$th component of $X_0$, and to the {\it word type} of an element of $G^{r_1}_{d_1}(X_0)$ as the string of numbers obtained writing the distinguished indices $i(1), i(2), \dots i(g)$ in succession. Clearly the word and the Young tableau associated with a given limit linear series are equivalent data.

\begin{thm}\label{thm:smoothing theorem}
The moduli scheme $\mathcal G^\circ$ is universally open and flat and reduced at every point $x \in \mathcal G^\circ_0$ corresponding to an inclusion of limit linear series on $X_0$ for which the ambient series is of word type $$\ubr{\big(12 \cdots (r_1+1)\big)}_{g/(r_1+1)\text{ times}}.$$
In particular, every such $x$ lifts to a unique inclusion of linear series on $X_\eta$.
\end{thm}
\begin{proof}
The results of Section~\ref{combinatorics_of_LLS} imply that $x$ is isolated in $\mathcal G^\circ_0$, and that both the ambient and included limit linear series of $x$ are refined. By \cite[Proposition 3.7]{osserman2015relative} it therefore suffices to show that (1) the map $\mathcal G^\circ\rightarrow B$ has universal relative dimension at least $0$ at $x$; and (2) $\mathcal G^\circ_0$ is reduced at $x$.

Proposition~\ref{prop:two constructions compatible} together with \cite[Corollary 5.1]{osserman2015relative} reduces item (1) to showing that $\mathcal G_D$ has dimension at least 1 at $x$. Accordingly, choose $D$ so that $|D_0|$ avoids the set of base points of $x$. We then have 
$$\dim \mathcal G^1_D=1+g+d+|V(\Gamma)|\big((r_2+1)(r_1-r_2)+(r_1+1)(d_1+\tilde d-g-r_1)\big).$$ 
It follows that each irreducible component of 
$\mathcal G^2_D$ has dimension at least
$$\dim \mathcal G^1_D-\sum_vd_v(r_1+1)=\dim \mathcal G^1_D-\tilde d(r_1+1).$$
%$$1+g+d+|V(\Gamma)|(r_2+1)(r_1-r_2)+|E(\Gamma)|(d_1+\tilde d-g-r_1)+(r_1+1)(d_1-g-r_1)$$
since it has codimension at most $\sum_vd_v(r_1+1)=\tilde d(r_1+1)$. Similarly $\widetilde{\mathcal G}^3_D$ has dimension at least 
$$\dim \mathcal G^1_D-\tilde d(r_1+1)- d |V(\Gamma)|(r_2+1)$$
%$$\dim \mathcal G^2_D-|V(\Gamma)|(r_2+1)\tilde d\geq 1+g+d+|V(\Gamma)|(r_2+1)(r_1-r_2- d)+|E(\Gamma)|(d_1+\tilde d-g-r_1)+(r_1+1)(d_1-g-r_1)$$
near $x$. On the other hand, since the included limit linear series corresponding to $x$ is refined, the linked chain of flags that cuts out $\widetilde{\mathcal G}_D$ inside $\widetilde {\mathcal G}^3_D$ is $(r_2+1,r_1+1)$-strict at $x$, and the codimension of $\widetilde{\mathcal G}_D$ inside $\widetilde {\mathcal G}^3_D$ is at most 
$$|E(\Gamma)|\big((r_2+1)(d_2+\tilde d-r_2-g)+(r_1-r_2)(d_1+\tilde d-r_1-g)\big)$$
near $x$ by Theorem~\ref{thm:dimension of linked determinantal loci}. A straightforward calculation now shows that $\widetilde{\mathcal G}_D$ has dimension at least $\rho_{g,r_1,d_1}+\mu+1$ near $x$, as does $\mathcal G_D$. So item (1) is proved.

We now prove item (2). To this end, let $( L_x,E_x,(V^2_{x,v},V^1_{x,v})_{v\in V(\Gamma)})$ be the special tuple represented by $x$. For each $v\in V(\Gamma)$ let $e_v$ denote the degree of $E_x$ in $Z_v$. According to Theorem~\ref{simplicity_of_base_points} of Section~\ref{combinatorics_of_LLS}, we have $e_v\in\{0,1\}$ for every $v$. Let $V^1(\Gamma)$ denote the set of vertices $\tilde v$ for which $e_{\tilde v}=1$. We know that $$x_1:=( L_x,(V^1_{x,v})_v)\in  G^{r_1}_{w_0}(X/B)_0=G^{r_1}_{w_0}(X_0)$$ is a refined limit $\mathfrak g^{r_1}_{w_0}$ on $X_0$. %Here  $G^{r_1}_{w_0}(X/B)_0=$ is the moduli space of limit $\mathfrak g^{r_1}_{d_1}$s on $X_0$ as in \cite[Definition 3.7]{osserman2014limit}.
But \cite[Proposition 4.2.6, Proposition 4.2.9]{osserman2014higherlimit} establishes that the moduli space $G^{r_1}_{w_0}(X_0)$ is (schematically) isomorphic to the Eisenbud--Harris moduli space of limit linear series over refined locus. In particular, $G^{r_1}_{w_0}(X/B)_0$ is reduced at $x_1$.

By construction, %since $\mathcal G^2_0$ is cut out in $\mathcal G^1_0$ by the same manner as the construction of $G^{r_2}_{w_0}(X_0)$, 
$\mathcal G^2_0$ is locally isomorphic near $x$ to the closed subscheme of 
$$\mathrm{Hilb}^\circ_d(X_0)\times \prod_{v\in V(\Gamma)}Gr(r_2,\widetilde V^1_{x,v})$$
cut out by the (refined) ramification condition on $\prod_{v\in V(\Gamma)}Gr(r_2,\widetilde V^1_{x,v})$ induced by $( L_x,(V^2_{x,v})_v)\in G^{r_2}_{w_0}(X_0)$, where $\widetilde V^1_{x,v}$ denotes the restriction of $V^1_{x,v}$ to $Z_v$. Let $\mathcal G^2_x$ be a neighborhood of $x$ in $\mathcal G^2_0$ that witnesses this (local) isomorphism. Then by construction for each $\tilde v\in V^1(\Gamma)$ and each point 
$y=(E_y,(\widetilde V^2_{y,v})_v)$ in $\mathcal G^2_x$, we have that $\widetilde V^2_{y,\tilde v}\subset \widetilde V^1_{x,\tilde v}$ contains a fixed vector $f_{x,\tilde v}\in \widetilde V^1_{x,\tilde v}$, determined by $x$.

On the other hand, any sufficiently small neighborhood $U_x$ of $x$ in $\mathcal G^1_0$ is locally isomorphic to
$$%U_x=\bigg(
\prod_{\tilde v\in V^1(\Gamma)}\mathrm{Hilb}^\circ_1(Z_{\tilde v})\times \prod_{v\in V(\Gamma)}Gr(r_2,\widetilde V^1_{x,v}).
%\bigg)\ \bigcap\ \mathcal G^1
$$
%contains a neighborhood of $x$ in $\mathcal G^1$, 
Here $\mathrm{Hilb}^\circ_1(Z_{\tilde v})=Z^\circ_{\tilde v}$. Thus we see that the vanishing locus of \eqref{eqn:G} on $U_x$ is contained in the vanishing locus of
\begin{equation}\label{eqn:Ux}
\varphi^*_v\mathscr V^2_{x,v}\hookrightarrow\widetilde \delta_{v*}  \delta^*_v L^v_x=H^0(Z_v, L^v_x)\otimes \mathscr O_{U_x}\rightarrow \widetilde\delta_{v*}(\delta^*_v L^v_x /(\delta^*_v L^v_x\otimes\sigma^*_v\mathcal I^\circ_v )).\end{equation}
Here $\mathscr V^2_{x,v}\subset \widetilde V^1_{x,v}\otimes \mathscr O_{Gr(r_2,\widetilde V^1_{x,v})}$ is the universal subbundle, $ L^v_x$ is the restriction of $( L_x)_{w_v}$ to $Z_v$, and $\mathcal I^\circ_v$ is the universal ideal sheaf on $\mathrm{Hilb}^\circ_{e_v}(Z_v)\times Z_v$. Note that the determinantal locus is compatible with pullback, and here the main way in which \eqref{eqn:G} differs from \eqref{eqn:Ux} is that we are now restricting the universal ideal sheaf $\mathcal I^\circ$ on $\mathrm{Hilb}^\circ_d(X_0)\times X_0$ to $Z_v$.  We have the following commutative diagram:
$$
\begin{tikzcd}
Z_v&U_x\times Z_v\lar{\delta_v}\dar\rar{\widetilde \delta_v}&U_x\dar{\sigma_v}\rar{\varphi_v}&Gr(r_2,V^1_{x,v}) \\
& \mathrm{Hilb}^\circ_{e_v}(Z_v)\times Z_v\rar{\widetilde \delta_v}&\mathrm{Hilb}^\circ_{e_v}(Z_v).
\end{tikzcd}
$$

Restricting \eqref{eqn:Ux} to $\mathcal G^2_x$, we see that for every $\tilde v\in V^1(\Gamma)$, the determinantal ideal always contains the condition imposed by vanishing of the image $\bar f_{x,\tilde v}\in H^0(\mathcal G^2_x,\widetilde\delta_{\tilde v*}(\delta^*_{\tilde v} L^{\tilde v}_x /(\delta^*_{\tilde v} L^{\tilde v}_x\otimes\sigma^*_{\tilde v}\mathcal I^\circ_{\tilde v} )))$ of $\widetilde\delta_{\tilde v*}\delta^*_{\tilde v}f_{x,\tilde v}$. Let $i_{\tilde v}\colon \mathrm{Hilb}^\circ_1(Z_{\tilde v})\hookrightarrow Z_{\tilde v}$ denote the natural inclusion, which induces a diagonal map $\Delta_{\tilde v}\colon \mathrm{Hilb}^\circ_1(Z_{\tilde v})\rightarrow  \mathrm{Hilb}^\circ_1(Z_{\tilde v})\times Z_{\tilde v}.$ We then have 

$$
\begin{tikzcd}\mathrm{Hilb}^\circ_1(Z_{\tilde v})\rar{i_{\tilde v}}&Z_{\tilde v}&\\
U_x\uar{\sigma_{\tilde v}}\rar{\Delta_{\tilde v}}\dar&U_x\times Z_{\tilde v}\uar{\delta_{\tilde v}}\dar{\sigma_{\tilde v}}\rar{\widetilde \delta_{\tilde v}}&U_x\dar{\sigma_{\tilde v}}\\ \mathrm{Hilb}^\circ_1(Z_{\tilde v})\rar{\Delta_{\tilde v}}
& \mathrm{Hilb}^\circ_1(Z_{\tilde v})\times Z_{\tilde v}\rar{\widetilde \delta_{\tilde v}}&\mathrm{Hilb}^\circ_1(Z_{\tilde v}).
\end{tikzcd}
$$
Since $\widetilde\delta_{\tilde v}\circ \Delta_{\tilde v}$ is the identity map on $\mathrm{Hilb}^\circ_1(Z_{\tilde v})$, and $\mathcal I^\circ_{\tilde v}$ is the ideal sheaf of $\Delta_{\tilde v}(\mathrm{Hilb}^\circ_1(Z_{\tilde v}))$, it follows that
$$\widetilde\delta_{\tilde v*}(\delta^*_{\tilde v} L^{\tilde v}_x /(\delta^*_{\tilde v} L^{\tilde v}_x\otimes\sigma^*_{\tilde v}\mathcal I^\circ_{\tilde v} ))=\Delta^*_{\tilde v}\delta^*_{\tilde v} L^{\tilde v}_x=\sigma^*_{\tilde v}i^*_{\tilde v} L^{\tilde v}_x$$
and $\bar f_{x,\tilde v}=\sigma^*_{\tilde v}i^*_{\tilde v}f_{x,\tilde v}$. Since $f_{x,\tilde v}$ cuts out a single point $P_{x,\tilde v}$ in $  Z^\circ_{\tilde v}$ by Theorem~\ref{simplicity_of_base_points}, we see that, locally at $x$, $\mathcal G^\circ_0$ is contained in $$\widetilde{\mathcal G}_x=\prod_{\tilde v\in  V^1(\Gamma)}\Big(P_{x,\tilde v}\times Gr(r_2,\widetilde V^1_{x,\tilde v}(-P_{x,\tilde v})\Big) \times \prod_{v\in V(\Gamma)\backslash V^1(\Gamma)}Gr(r_2,\widetilde V^1_{x,v}).$$

Now imposing the ramification condition that defines $\mathcal G^2_0$ inside $\widetilde {\mathcal G}_x$, we get a product over $\tilde v\in V^1(\Gamma)$ (resp., $v\in V(\Gamma)\backslash V^1(\Gamma)$) of intersections of pairs of Schubert varieties in $Gr(r_2,\widetilde V^1_{x,\tilde v}(-P_{x,\tilde v}))$ (resp., $Gr(r_2,\widetilde V^1_{x,v})$), %. This gives the local structure of $\mathscr G$ at $x$
that is zero-dimensional. According to \cite[Theorem 4.3.1]{kreiman2004richardson}, it is reduced;  and locally it contains $\mathcal G^\circ_0$. Hence $\mathcal G^\circ_0$ is reduced at $x$.
\end{proof}

\section{Combinatorics of inclusions of limit linear series}\label{combinatorics_of_LLS}
In this section we will describe (and partially implement) a general algorithm for computing the number of secant planes on a general curve when both of the basic dimension parameters $\rho$ and $\mu$ are zero. To do so, we count inclusions of limit series
\begin{equation}\label{basic_inclusion}
g^{s-d+r}_{m-d}+ p_1+ \dots+ p_d \hra g^s_m
\end{equation}
along a general chain $X= E_1 \cup \dots \cup E_g$ of elliptic curves whenever $\rho(g,s,m)=\mu(d,r+s-d,s)=0$ and the ambient $g^s_m$ is of word type $\ubr{(12 \cdots (s+1))}_{\text{u times}}$, where $u \geq 1$.\footnote{Note the change of variables relative to the preceding sections: $s=r_1$ and $s-d+r=r_2$. Recall that $d=d_1-d_2$.} Here the fact that $\mu=0$ implies that $1 \leq r \leq s-1$. The extremal cases $r=1$ and $r=s-1$ are distinguished, %(note that both of these generalize the case of double points of plane models), 
and accordingly we will pay particular attention to these.

\begin{defn}\label{type_A_inclusion}
A {\it good} inclusion of limit linear series on a curve $X$ of compact type is an inclusion \eqref{basic_inclusion} for which 
\begin{enumerate}
\item No base point $p_i$ lies along a point of attachment of $X$; and
\item At most one base point $p_i$ of the included $(s-d+r)$-dimensional series lies along any given component.
\end{enumerate}
\end{defn}
\begin{si}
In the remainder of this section, we will always assume $X= E_1 \cup \dots \cup E_g$ is a general chain of elliptic curves and $\rho=\mu=0$. We mark each component $E_j$, $j=1,\dots,g$ in two points $p$ and $q$, whose difference is nontorsion; we always mark components in their points of attachment with neighboring components.
\end{si}
Our first aim is to show that whenever the ambient $g^s_m$ is of word type $\ubr{(12 \cdots (s+1))}_{\text{u times}}$, all inclusions \eqref{basic_inclusion} are good. We will first show that inclusions that satisfy item (1) in Definition~\ref{type_A_inclusion} automatically satisfy item (2). In order to do so, we will make crucial use of the following technical device.

\begin{defn}\label{shift}
Let $E$ denote a smooth curve marked in distinct (fixed) points $p$ and $q$, and assume $g^{r_2}_{d_2} \hookrightarrow g^{r_1}_{d_1}$ is an inclusion of linear series on $E$. 
The vanishing sequences at $p$ and $q$ of the included $g^{r_2}_{d_2}$ along $E$ are subsequences $a_{v(p)}=(a_{v(0)},\dots,a_{v(r_2)})$ and $b_{w(q)}=(b_{w(0)},\dots,b_{w(r_2)})$ of the vanishing sequences $a_p=(a_0,\dots,a_s)$ and $b_q=(b_0,\dots,b_s)$ at $p$ and $q$, respectively, of the ambient $g^{r_1}_{d_1}$. The corresponding {\it shift} of the included $g^{r_2}_{d_2}$ along $E$ is $\sum_{i=0}^{r_2} (w(i)-v(i))$.
\end{defn}

Shifts in the context of inclusions of limit linear series play a role analogous to that of the monotonicity property systematically exploited by Eisenbud and Harris %satisfied by a limit linear series along a chain of smooth curves 
which establishes, roughly, the existence of compatible bases for the aspects of a limit series along a chain of smooth curves whose vanishing orders increase as one moves along the chain; see, e.g. \cite[Lem 5.2]{MOCURVE}. Shifts quantify how the inclusions of aspects along the components change as one moves ``downward" (i.e., in the monotone direction) along the chain. That is, the vanishing orders of the aspects of the included series, thought of as vanishing orders of sections of the ambient series {\it shift} relative to their positions at the top of the chain. In particular, monotonicity implies that there is a natural notion of {\it cumulative shift} for the included series in \eqref{basic_inclusion} along a chain of elliptic curves, obtained by summing the shifts of its aspects along the individual components.

\begin{rem}\label{basic_heuristic} 
Given an inclusion \eqref{basic_inclusion} of limit linear series along an elliptic chain, because each index of the included $g^{s-d+r}_m$ can shift by at most $s-(s-d+r)=d-r$ places, the cumulative shift of the $g^{s-d+r}_m$ along the chain is bounded above by $(s-d+r+1)(d-r)$, which in turn is equal to $d(s-d+r)$ since we assume $\mu(d,r,s)=0$. In particular, we would expect the presence of a base point along an elliptic component to impose $s-d+r$ shifts on our included $g^{s-d+r}_m$.
\end{rem}

\begin{rem}\label{series_on_ell_curves}
In studying how base points affect the vanishing sequences of the included $g^{s-d+r}_m$, viewed as subsequences of the ambient limit series $g^s_m$, we will crucially use the fact that the vanishing sequences of the aspects of a limit series of type $\ubr{(12 \cdots (s+1))}_{\text{u times}}$ at marked points along the elliptic chain are nearly increasing sequences of consecutive integers of length $(s+1)$.
\end{rem}

\begin{defn}\label{nearly-consecutive_sequence}
A {\it nearly-consecutive} increasing sequence of integers of length $\ell \geq 1$ is a sequence of integers $a_1, a_2, \dots, a_{\ell}$ for which every $a_j-a_{j-1}=1$ for every $j \neq j_0$ for some $2 \leq j_0 \leq \ell$, while $a_{j_0}-a_{j_0-1}=2$. We call $j_0$ the {\it distinguished index}.
\end{defn}

\begin{lem}\label{nearly_consecutive_lemma}
The vanishing sequences of the aspects of the limit series $g^s_m$ of type $\ubr{(12 \cdots (s+1))}_{\text{u times}}$ at the points of attachment of the elliptic chain $X$ are nearly-consecutive sequences.
\end{lem}

\begin{proof}
Let $E_j$ denote the $j$th elliptic component of $X$. If either $j=1$ or $j=g$, the desired conclusion is clear: indeed, the vanishing sequences at the marked points along $E_1$ and $E_g$ are $(0,1,\dots,s)$ and $(m-s-1, m-s, \dots, m-2,m)$, irrespective of the word type of the $g^s_m$. So assume $1<j<g$. Monotonicity of the $g^s_m$ implies that the vanishing sequences of the aspects $a_p=(a_0,\dots,a_s)$ in the ``first" marked points $p$ along each component comprise an increasing sequence of $(s+1)$-tuples. More precisely, letting $p(j)$ and $p(j+1)$ denote the first marked points along $E_j$ and $E_{j+1}$, respectively, the vanishing sequence $a_{p(j+1)}$ is derived from $a_p$ by adding one in every coordinate except one for which the vanishing order remains fixed (here we apply $\rho(g,s,m)=0$). The assumption that $g^s_m$ is of type $\ubr{(12 \cdots (s+1))}_{\text{u times}}$ guarantees that the corresponding index $i(j)$ is exactly $(j-1) \text{ mod } (s+1)$, and that $a_{p(j)}$ is the unique nearly-consecutive increasing sequence with distinguished index $i(j)$ at which the corresponding vanishing order is $2 (j-1) \text{ mod } (s+1)$.
%$a_{p(j)}=\lfloor \frac{j-1}{s+1} \rfloor s+ (0,1, \dots, 2\lbrace \frac{j-1}{s+1}\rbrace \text{ mod } (s+1)-2, j \text{ mod } (s+1))$.

%If either $j=1$ or $j=g$, the desired conclusion is clear: indeed, the vanishing sequences at the marked points along $E_1$ and $E_g$ are $(0,1,\dots,s+1)$ and $(m-s-1, m-s, \dots, m-2,m)$, irrespective of the word type of the $g^s_m$. So assume $1<j<g$.
\end{proof}

\begin{prop}\label{simplicity_of_base_points}
At most one (simple) base point $p_i$ of the included $(s-d+r)$-dimensional series in a good inclusion \eqref{basic_inclusion} lies along any given elliptic component.
\end{prop}

\begin{proof}
Say that $M$ base points (possibly with multiplicities) lie along the component $E_j$.  
As in Definition~\ref{shift} above, the vanishing sequences at $p$ and $q$ of the included $g^{s-d+r}_{m}$ along $E_j$ are subsequences $a_{v(p)}=(a_{v(0)},\dots,a_{v(s-d+r)})$ and $b_{w(q)}=(b_{w(0)},\dots,b_{w(s-d+r)})$ of the vanishing sequences $a_p=(a_0,\dots,a_s)$ and $b_q=(b_0,\dots,b_s)$ at $p$ and $q$, respectively, of the ambient $g^s_m$.
We now claim that, if $M>1$, the corresponding shift of the included series satisfies $\sum_{i=0}^{s-d+r} (w(i)-v(i)) \geq 1+M(s-d+r)$.
%Because $\mu=0$, the cumulative shift for the included series along the elliptic chain is at most $d(s-d+r)$, so 
In light of Remark~\ref{basic_heuristic}, the claim will imply that $M \leq 1$. Further, whenever $j=1$ or $j=g$, it is easy to see that a shift of $(M-1)+M(s-d+r)$ vanishing orders along the component is forced, so hereafter we assume $1<j<g$. In this case, Lemma~\ref{nearly_consecutive_lemma} reduces us to proving the following statement about nearly-consecutive sequences.

\begin{claim}\label{nearly_consecutive_sequences}
Fix a choice of positive integers $s<m$ and let $i_0$ and $M$ be positive integers for which $M \geq 2$ and $2 \leq i_0 \leq s$. Let ${\rm L}(i_0,s,m)$ denote the unique increasing nearly-consecutive sequence of nonnegative integers of length $s+1$ beginning in 0 and ending in $s$ with distinguished index $i_0$, and let ${\rm L}^{\circ}={\rm L}^{\circ}(i_0,s,m)$ denote the {\it $m^{\circ}$-complementary} sequence defined by ${\rm L}^{\circ}_i=m-1-{\rm L}_i$ for all $i \neq i_0$, and ${\rm L}^{\circ}_{i_0}=m-{\rm L}_{i_0}$. Further choose an increasing subsequence ${\rm L}^*$ of ${\rm L}$ with length $s^* \leq s$. Then every decreasing subsequence ${\rm Q}$ of ${\rm L}^{\circ}$ for which
\[
{\rm L}^*+ {\rm Q} \leq (m-M-1,\dots,m-M-1,m-M)
\]
is obtained from the $m^{\circ}$-complement $({\rm L}^*)^{\circ}$ by shifting in at least $M s^*+1$ places.
\end{claim}

Note that the assumption that ${\rm L}$ begin in 0 merely represents a convenient choice of normalization, and corresponds to removing the base points concentrated in $p$ of the ambient $g^s_m$. 
\begin{proof}[Proof of Claim~\ref{nearly_consecutive_sequences}]
We treat first the case $M \geq s^*+3$,
which geometrically corresponds to the situation in which the number $M$ of base points is large relative to the dimension $s^*=s-d+r$ of the included series. The point is that in this regime, we have $(M-1)s^*+ M-2 \geq M s^*+1$, so it suffices to show that $(M-1)s^*+ M-2$ shifts are forced. %Estimating the number of shifts that are forced hinges, in turn, on the placement of the distinguished index $i_0$ relative to $\wt{\rm L}$.

To do so, assume that the distinguished index $i_0$ does not belong to ${\rm L}^*$. Then ${\rm L}^*_i+ ({\rm L}^*)^{\circ}_i=m-1$ for all $i$. It is clear that in passing from $({\rm L}^*)^{\circ}_i$ to ${\rm Q}_i$, at least $M-2$ shifts are required, irrespective of how large $M$ is relative to $i_0-{\rm L}^*_i$. But in fact at most {\it one} $({\rm L}^*)^{\circ}_i$ may shift by exactly $M-2$ places. The analysis when $i_0$ belongs to ${\rm L}^*$ is analogous (and easier); we conclude immediately.
%it suffices to show that at least one $\wt{\rm L}^{\circ}_i$ necessarily shifts by $M$ places.
%What we will show, then, is that each shift of $\wt{\rm L}^{\circ}_i$ by $M-2$ places is counterbalanced by a shift of some other $\wt{\rm L}^{\circ}_{i^{\pr}}$ by $M$ places. To this end, let $d_i:=i_0-\wt{\rm L}_i$, $i=1,\dots,s+1$, and assume that $d_{i_1} \leq M \leq d_{i_2}$, where by convention we let $d_{i_1}=-\infty$ (resp., $d_{i_2}=\infty$) if $M$ is less than (resp., greater than) all of the differences $d_i$. Then ${\rm L}_i$ necessarily shifts at least

%In passing from $\wt{\rm L}^{\circ}_i$ to ${\rm Q}_i$, at least $M-1$ shifts are required, irrespective of how large $M$ is relative to $i_0-\wt{\rm L}_i$, so we conclude easily in this case.
Now say $2 \leq M \leq s^*+2$. Given an arbitrary increasing length-$(s^*+1)$ subsequence ${\rm L}^*$ of ${\rm L}$, let
${\rm S}_1= {\rm S}_1({\rm L}^*)$ denote the collection of elements of ${\rm L}^*$ lying either to the right of ${\rm L}_{i_0}$, or at least $M$ places to the left of ${\rm L}_{i_0}$, and let ${\rm S}_2$ denote the complement of ${\rm S}_1$ in ${\rm L}^*$. Let $a^*=a^*({\rm L}^*)$ denote the cardinality of ${\rm S}_1$.
%$a^*=a^*({\rm L}^*)$ denote the number of elements in ${\rm L}^*$ that are either strictly greater than ${\rm L}^*_{i_0}$, or lie strictly more than $M$ places to the left of ${\rm L}^*_{i_0}$. 
Note that every element of ${\rm S}_1$ (resp., ${\rm S}_2$) except for possibly one contributes $M$ forced shifts; the exceptional element, if it exists, contributes $M-1$ (resp., $M-2$) shifts. Because there is at most one exceptional element, it follows that ${\rm L}^*$ is associated with at least $\nu(a^*):=(M-1)(s^*+1-a^*)+ Ma^*-1$ forced shifts, and we may conclude provided $a^* \geq s^*+3-M$. Indeed, $\nu$ is an increasing function of $a^*$, and by rewriting we see that $\nu(s^*+3-M)= Ms^*+1$. 

It remains to handle those cases in which $a^* < s^*+3-M$. Note that $\# {\rm S}_2 \leq M$ by construction, so $a^* \geq s+1-M$ is automatic. %as there can be at most $M$ elements of ${\rm L}^*$ lying within $M$ places to the left of ${\rm L}_{i_0}$. 
Accordingly there are two basic situations, depending upon whether $a^*=s^*+1-M$ or $a^*=s^*+2-M$. In the first situation, the $M$ elements ${\rm L}_{i_0-j+1}$, $j=1, \dots, M$ belong to ${\rm L}^*$. %Then at least $(s^*+1-M)M$ shifts are forced by the $s^*+1-M$ elements in ${\rm L}^* \setminus \{\rm L_{i_0-M+1}, \dots, L_{i_0}\}$, and it suffices to show that the remaining $M \geq 2$ elements force at least $M(M-1)+1$ shifts. 
More precisely, 
\[
{\rm S}_2=\{{\rm L}_{i_0-M+1}, \dots, {\rm L}_{i_0}\}= \{i_0-M, \dots, i_0-2,i_0\}.
\]
%As usual, the shift sequences induced by ${\rm S}_1$ and ${\rm S}_2$ depend upon location of the element of ${\rm L}^*$ corresponding to the unique index $i$ for which ${\rm L}^*_i+ Q_i=m-M$. If this exceptional element belongs to ${\rm S}_2$ and 
It is easy to check that ${\rm S}_1$ induces at least $(s^*+1-M)M-\epsilon$ shifts, while ${\rm S}_2$ minimally induces $M(M-1)+1+\epsilon$ shifts, where the value of $\epsilon$ is 1 (resp., 0) when distinguished index of $L^*$ belongs to ${\rm S}_1$ (resp, ${\rm S}_2$).

The case in which $a^*=s^*+2-M$ is related to the preceding case by exchanging an element of ${\rm S}_2$ for an element of ${\rm S}_1$. It is easy, if tedious, to see that in so doing the number of induced shifts does not decrease.
\end{proof}
%Let $d_i:=i_0-\wt{\rm L}_i$, $i=1,\dots,s+1$, and assume that $d_{i_1} \leq M \leq d_{i_2}$, where by convention we let $d_{i_1}=-\infty$ (resp., $d_{i_2}=\infty$) if $M$ is less than (resp., greater than) all of the differences $d_i$. Then ${\rm L}_i$ necessarily shifts at least
%Rewriting, we find that $(M-1)(s^*+1-a)+ Ma= M \cdot s^*+1$. 
%So Claim~\ref{nearly_consecutive_sequences} follows immediately from Subclaim~\ref{shift-minimizing_subsequences}.
%\end{proof}

%Without loss of generality, we may assume that $i_0$ lies between the smallest and largest indices of $v(p)=v(q)$, and that the smallest index $v_0(p)$ satisfies $v_0(p) \geq i_0-M$, since otherwise the claim is obvious. Set $M_0:=i_0-v_0(p)$. The presence of a multiplicity-$M_0$ base point clearly forces $M_0(s-d+r+1)$ shifts, by construction, and it suffices to show that the ``leftover" multiplicity-$(M-M_0)$ base point forces $(M-M_0)(s-d+r)-M_0+1)$ further shifts. But again it suffices to consider 
%depending upon whether or not the distinguished index $i_0$ belongs to $v(p)=v(q)$.

%Clearly $M-1+M(s-d+r)$ is strictly greater than $M(s-d+r)$ whenever $M>1$. 
Proposition~\ref{simplicity_of_base_points} follows immediately from Claim~\ref{nearly_consecutive_sequences}.
\end{proof}

\begin{lem}\label{refinedness_of_included_series}
The included $(s-d+r)$-dimensional series in any good inclusion \eqref{basic_inclusion} is {\it refined}.
\end{lem}

\begin{proof}
Because our ambient $g^s_m$ is refined by construction, refinedness of the included $g^{s-d+r}_m$ amounts to the statement that the sum of the vanishing sequences of the included $(s-d+r)$-dimensional aspect is always maximal possible, i.e. either %$(s-d+r-2,\dots,s-d+r-2,s-d+r-1)$ or $(s-d+r-1,\dots,s-d+r-1,s-d+r)$  
$(m-2,\dots,m-2,m-1)$ or $(m-1,\dots,m-1,m)$ depending upon whether or not a (simple) base point lies along a given component. It also clearly suffices to prove maximality of vanishing on components containing base points $p_i$. Accordingly we may argue by induction on the index $1 \leq i \leq d$ of the given base point. Since the assertion is clear when $i=1$, it suffices to show that maximality-of-vanishing is preserved in the presence of a base point. Much as in the proof of Proposition~\ref{simplicity_of_base_points}, this in turn reduces to a purely combinatorial statement about nearly-consecutive sequences of integers that we leave to the reader.
\end{proof}
%Without loss of generality, we may assume that $i_0$ lies between the smallest and largest indices of $v(p)=v(q)$, and that the smallest index $v_0(p)$ satisfies $v_0(p) \geq i_0-M$, since otherwise the claim is obvious. Set $M_0:=i_0-v_0(p)$. The presence of a multiplicity-$M_0$ base point clearly forces $M_0(s-d+r+1)$ shifts, by construction, and it suffices to show that the ``leftover" multiplicity-$(M-M_0)$ base point forces $(M-M_0)(s-d+r)-M_0+1)$ further shifts. But again it suffices to consider 
%depending upon whether or not the distinguished index $i_0$ belongs to $v(p)=v(q)$.

%Clearly $M-1+M(s-d+r)$ is strictly greater than $M(s-d+r)$ whenever $M>1$. 
\begin{lem}\label{uniquely_prescribed_base_pt}
The position of a base point $p_i$ of (the included $g^{s-d+r}_m$ of) any good inclusion \eqref{basic_inclusion} along an elliptic component $E_j$ is uniquely prescribed.
\end{lem}

\begin{proof}
%A slight refinement of the argument used in proof of Lemma~\ref{simplicity_of_base_points} above shows that the sum of the vanishing sequences of the included $(s-d+r)$-dimensional aspect is always maximal possible, i.e. either %$(s-d+r-2,\dots,s-d+r-2,s-d+r-1)$ or $(s-d+r-1,\dots,s-d+r-1,s-d+r)$ $(m-2,\dots,m-2,m-1)$ or $(m-1,\dots,m-1,m)$ depending upon whether or not a base point lies along a given component. 
The maximality-of-vanishing property established in the proof of Lemma~\ref{refinedness_of_included_series} implies, in particular, that there is always a pair of {\it aligned} vanishing orders $(a_p(k),a_q(k))$ of the included series along each component containing a base point $p_i$. Letting $\mc{O}(\al p+ \be q)$ denote the line bundle underlying the aspect of our limit linear series along $E_j$, it then follows that the degree 0 line bundle $\mc{O}(\al p+ \be q-a_p(k)p -a_q(k)q- p_i)$ has a nonzero global section $\sig_j$. This, in turn, may only happen if $p_i$ is linearly equivalent to $(\al-a_p(k)) p+(\be-a_q(k))q$.
\end{proof}

\begin{rem}
The global section $\sig_j$ of the aspect line bundle on $E_j$ singled out in the proof of Lemma~\ref{uniquely_prescribed_base_pt} is a key ingredient in the proof of our smoothing theorem~\ref{thm:smoothing theorem} for inclusions of limit linear series in which the ambient series is of type $\ubr{(12 \cdots (s+1))}_{\text{u times}}$.
\end{rem}

\begin{prop}\label{points_of_attachment_avoidance}
Assume that $X$ is a general chain of $g$ elliptic curves. The {\it only} inclusions of limit linear series \eqref{basic_inclusion} for which the ambient $g^s_m$ is of combinatorial type $\ubr{(12 \cdots (s+1))}_{\text{u times}}$ are good.
\end{prop}

\begin{proof}
We will show that no inclusions \eqref{basic_inclusion} exist for which the ambient $g^s_m$ is of combinatorial type $\ubr{(12 \cdots (s+1))}_{\text{u times}}$ and some base point $p_i$ of the included series $g^{s-d+r}_m$ is supported at a node of $X$. For this it suffices to show that no inclusions \eqref{basic_inclusion} exist along a curve $\wt{X}$ obtained from $X$ via blow-ups in the nodes for which the ambient $g^s_m$ is of combinatorial type $\ubr{(12 \cdots (s+1))}_{\text{u times}}$ and some $p_i$ belongs to the {\it interior} of a rational component of $\wt{X}$. Just as in the proof of Proposition~\ref{simplicity_of_base_points}, it suffices to show that the placement of $M$ putative base points (with multiplicities) along a rational component imposes at least $M(s-d+r)+1$ shifts among indices of the included $g^{s-d+r}_m$. And this, in turn, may be distilled to a claim about nearly-consecutive sequences, as follows.

\begin{claim}\label{nearly_consecutive_sequences_bis}
Fix a choice of positive integers $s<m$ and let $i_0$ and $M$ be positive integers for which $M \geq 2$ and $2 \leq i_0 \leq s$. Let ${\rm L}(i_0,s,m)$ denote the unique increasing nearly-consecutive sequence of nonnegative integers of length $s+1$ beginning in 0 and ending in $s$ with distinguished index $i_0$, and let ${\rm L}^{\circ \circ}={\rm L}^{\circ \circ}(i_0,s,m)$ denote the {\it $m^{\circ \circ}$-complementary} sequence defined by ${\rm L}^{\circ \circ}_i=m-{\rm L}_i$ for all $i$. Further choose an increasing subsequence ${\rm L}^*$ of ${\rm L}$ with length $s^* \leq s$. Then every decreasing subsequence ${\rm Q}$ of ${\rm L}^{\circ \circ}$ for which
\[
{\rm L}^*+ {\rm Q} \leq (m-M,\dots,m-M)
\]
is obtained from the $m^{\circ \circ}$-complement $({\rm L}^*)^{\circ \circ}$ of ${\rm L}^*$ by shifting in at least $M s^*+1$ places.
\end{claim}

\begin{proof}[Proof of Claim~\ref{nearly_consecutive_sequences_bis}]
The proof follows the same lines as the proof of Claim~\ref{nearly_consecutive_sequences}. It is clear, first of all, that every element of ${\rm L}^*$ contributes {\it at least} $M-1$ shifts, so %whenever $M \geq s^*+2$ the desired conclusion follows from the fact that $(s^*+1)(M-1) \geq Ms^*+1$.
the issue is controlling the number of elements that contribute {\it exactly} $M-1$ of these. But in fact there can be at most $M-1$ of these, corresponding to elements of ${\rm L}$ that lie strictly less than $M$ places to the left of ${\rm L}_{i_0}$. Accordingly, we see that at least
$(M-1)(M-1)+M(s^*+1-(M-1))=Ms^*+1$ shifts are forced.
\end{proof}

Proposition~\ref{points_of_attachment_avoidance} follows immediately from Claim~\ref{nearly_consecutive_sequences_bis}.
\end{proof}

\subsection{The case $r=1$}
Assume that $X$ is a general chain of $g$ elliptic curves, that the ambient $g^s_m$ is of type $\ubr{(12 \cdots (s+1))}_{\text{u times}}$, and that $r=1$. Solving $\rho=\mu=0$ explicitly, we obtain the useful parameterization
\[
d=t+1, s=2t, g=(s+1)u, \text{ and } m=s(u+1)
\]
where $t$ and $u$ are positive integers. Rewritten in terms of $t$ and $u$, our basic inclusion \eqref{basic_inclusion} reads
\begin{equation}\label{basic_inclusion_r=1}
g^t_{2t(u+1)-(t+1)}+ p_1+ \cdots+ p_{t+1} \hra g^{2t}_{2t(u+1)}.
\end{equation}
Note that our re-usage of the $u$-variable here is consistent. We may characterize the set of good inclusions on $X$ as follows. %By leveraging Lemmas~\ref{simplicity_of_base_points} and \ref{uniquely_prescribed_base_pt} above, we arrive at the following characterization of the 

%\medskip
%From now on, we will {\it assume} the following holds:
\begin{thm}\label{type_A_inclusions} %Irrespective of the combinatorial type of the ambient series $g^{2t}_{2t(u+1)}$, 
The set of inclusions \eqref{basic_inclusion_r=1} on $X$ is indexed by the elements of the set
\[
{\rm S}=\{(j_1,\dots,j_{t+1}): 1 \leq j_1 < \cdots < j_{t+1} \leq (2t+1)u \text{ and } j_k \not\equiv 2k-1 \hspace{3pt} (\text{mod }2t+1) \text{ for all } k=1,2,\dots,t+1\}.
\]
\end{thm}

\begin{proof}
The result follows easily from the assertion that when $r=1$, {\it the index of the unique aligned section of the included aspect $g^{s-d+r}_m$ along the $j_k$th component (along which the $k$th base point $p_k$ lies, by definition) is $(2k-1)$ mod $2t+1$}.
To see this, note that when $r=1$, %each index of the {\it included} series $g^{s-d+r}_m$ corresponds to the aligned section precisely once. On the other hand, 
the index of the aligned section, viewed as a section of the ambient series, necessarily increases by precisely 2 along each component containing a base point of the included series. But because the key assertion is clear when $k=1$, by induction it also holds in the case of arbitrary $k$.
\end{proof}

%The upshot of Theorem~\ref{type_A_inclusions} is that we can write a combinatorially-useful formula for the cardinality of $S_A$ whenever the ambient $g^{2t}_{2t(u+1)}$ has word type ${\bf w}=\ubr{(12 \cdots 2t+1)}_{\text{u times}}$. Namely, in that case the description of $S_A$ simplifies to
%\[
%S_A= \{(j_1,\dots,j_{t+1}): 1 \leq j_1 < \cdots < j_{t+1} \leq (2t+1)u \text{ and } j_k \not\equiv 2k-1 \text{ (mod }2t+1)\}.
%\]
\begin{rem}\label{good_inclusion_set}
We expect more generally that when $r=1$, $\rho=\mu=0$, and our ambient series is of word type ${\bf w}$, the set of {\it good} inclusions is indexed by
\[
{\rm S}({\bf w})=\{(j_1,\dots,j_{t+1}): 1 \leq j_1 < \cdots < j_{t+1} \leq (2t+1)u \text{ and } w_{j_k} \neq 2k-1, j=1,2,\dots,t+1\}.
\]
\end{rem}

\subsection{Graphical representation of the set of good inclusions}
A useful schematization of the set ${\rm S}$ in Remark~\ref{good_inclusion_set} when ${\bf w}= \ubr{(12 \cdots (s+1))}_{\text{u times}}$ is as follows. Let ${\rm G}={\rm G}({\rm S})$ denote a $d \times g$ (i.e. $(t+1) \times (2t+1)u$) grid of (vertices labeled by) positive integers between 1 and $s+1$, whose $j$th row corresponds to the (placement of) the $j$th base point, and in which each row the (same) word ${\bf w}= \ubr{(12 \cdots (s+1))}_{\text{u times}}$ is written. Now trace edges through every pair of vertices in distinct and adjacent rows.
The cardinality of ${\rm S}$ is then equal to a number of paths through the resulting graph (which we continue to label by) ${\rm G}$, subject to a natural positivity condition, as follows. 

\begin{defn}\label{traversal}
A {\it traversal} $\mc{T}$ of ${\rm G}$ is a connected subgraph of ${\rm G}$ that passes through precisely one vertex in each row. We make $\mc{T}$ into a directed graph by orienting each edge in the ``downward" direction. A traversal $\mc{T}$ is {\it positive} if, moreover, each of its edge is obtained by a movement that is down (by a single unit) and to the right (by a nonzero number of units) within the grid.
\end{defn}

\begin{rem}
The cardinality of $\#{\rm S}$ is equal to the number of positive traversals of ${\rm G}$.
\end{rem}

%\subsubsection{A combinatorial recipe, through examples}
%{\bf \flu Example one: $s=d=2$.} 
\begin{ex}
Say $s=d=2$. The corresponding grid ${\rm G}$ (minus the edges linking vertices in the first and second rows) is given by
{\small
\[
\begin{matrix}
\ast & 2 & 3 & \ast & 2 & 3 & \cdots & \ast & 2 & 3 \\
1 & 2 & \ast & 1 & 2 & \ast & \cdots & 1 & 2 & \ast
\end{matrix}
\]
}
in which asterisks denote positions that are prohibited by our positivity condition. %The number of positive traversals is clearly the number of positive traversals in the following reduced version of ${\rm G}$:
%\[
%\begin{matrix}
%\begin{array}{ccccccccccccccccc}
%\ast & 2 & 3 & \ast & 2 & 3 & \cdots & \ast & 2 & 3 & \ast & \ast & \ast \\
%\ast & \ast & \ast & 1 & 2 & \ast & \cdots & 1 & 2 & \ast & 1 & 2 & \ast 
%\end{array}
%\end{matrix}
%\]
Each traversal may be thought of as a length-2 word, namely either $2 1$, $2 2$, $3 1$, or $3 2$. It is clear, furthermore, that the number of traversals is independent of the length-2 word chosen. So it suffices to compute the number ${\rm N}(2,u-1)$ of positive traversals associated with a {\it fixed} choice of length-2 word. The quantity ${\rm N}(2,u-1)$, in turn, computes the number of {\it top-to-bottom paths} in a directed graph $\Ga(2,u-1)$, which is a special case of a more general construction that we present next. It follows easily that the number of positive traversals of ${\rm G}$ is $4 \binom{u}{2}$, which is the predicted value.
\end{ex}

\begin{const} Given positive integers $d$ and $e$, we construct a graph $\Ga(d,e)$ as follows.
\begin{itemize}
\item[(i)] The vertices of $\Ga(d,e)$ are the vertices of a $d \times e$ grid with integer-valued coordinates $(i,j)$, where $1 \leq i \leq d$ and $1 \leq j \leq e$. Each vertex is labeled by its row, i.e., by its first coordinate.
\item[(ii)] Edges in $\Ga(d,e)$ link only vertices in distinct and adjacent rows, and an edge links $(i,j)$ with $(i+1,k)$ if and only if $k \geq j$.
\item[(iii)] We orient $\Ga(d,e)$ according to the convention that smaller labels point towards larger labels, i.e. each edge is oriented ``from top to bottom".
\end{itemize}
\end{const}

\noindent For example, $\Ga(2,4)$ is the following graph (the downwards orientation of every edge is omitted):

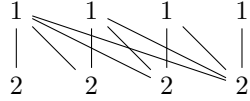
{\hspace{90pt}
\begin{tikzpicture}[scale=.5]
\node (n1) at (1,10) {1};
\node (n2) at (3,10) {1};
\node (n3) at (5,10) {1};
\node (n4) at (7,10) {1};

\node (n5) at (1,8) {2};
\node (n6) at (3,8) {2};
\node (n7) at (5,8) {2};
\node (n8) at (7,8) {2};

\foreach \from/\to in {n1/n5, n1/n6, n1/n7, n1/n8, n2/n6, n2/n7, n2/n8, n3/n7, n3/n8, n4/n8} \draw (\from) -- (\to);
\end{tikzpicture}
}

\begin{defn}\label{top-to-bottom path}
A {\it top-to-bottom} path in $\Ga(d,e)$ is a connected subgraph of $\Ga(d,e)$ that passes through precisely one vertex in each row, which we orient downwards along each edge, i.e. from top to bottom.
\end{defn}

\begin{notation}
Let $N(d,e)$ denote the number of top-to-bottom paths in $\Ga(d,e)$.
\end{notation}

\begin{rem}
Clearly, $s^d$ is the number of possible $d$-tuples $(\al_1,\dots,\al_d)$ that label possible words corresponding to traversals of the $d \times s$ grid ${\rm G}$ introduced at the beginning of this subsection. The particularity of the $s=d=2$ case is that the number of possible positive traversals is {\it independent} of the choice of $(\al_1,\al_2) \in \{(2,1),(2,2),(3,1),(3,2)\}$, but this is not true in general.
On the other hand, it is easy to see that the number of top-to-bottom paths in the smaller graph $\Ga(d,e)$ is given by
\begin{equation}\label{N_d,e}
{\rm N}(d,e)= \binom{d+e-1}{d}.
\end{equation}
Indeed, \eqref{N_d,e} follows from the facts that the path counts $N(d,e)$ satisfy the binomial-type recursion
\[
{\rm N}(d,e)= {\rm N}(d-1,e)+ {\rm N}(d,e-1)
\]
whenever $d,e \geq 2$, and that
\[
{\rm N}(1,e)=e \text{ and } {\rm N}(d,1)=1.
\]
\end{rem}
The essential question, then, is understanding how positive traversals of ${\rm G}$ are stratified according to their associated graphs $\Ga(d,e)$, i.e., which values $e$ are possible. To do so, we introduce an auxiliary graph directed $\Ga^s(d)$. Here $\Ga^s(d)$ has integer coordinate-valued vertices $(i,j)$ where $1 \leq i \leq d$ and $i \leq j \leq s+i-1$; as usual, we equip $\Ga^s(d)$ with the orientation that has all edges pointing downward. Vertices with coordinates $(i,j)$ are labeled by $1+(i+j-1) \text{ mod }(s+1)$. Among the set of all top-to-bottom paths in $\Ga^s(d)$, starting from a vertex in the first (top) row and ending in a vertex in the bottom ($d$th) row, those that involve precisely $k$ {\it backwards} edges (i.e., downward and to the left) are indexed by $d$-tuples with associated graphs $\Ga(d,u-1-k)$. In general, we have $0 \leq k \leq d-2$. Putting all of this together, we conclude the following. %(assuming all inclusions of LLS are of multiplicity one).

\begin{thm}\label{good_inclusions_r=1} When $\rho=\mu=0$ and $r=1$, the number of inclusions \eqref{basic_inclusion_r=1} for a fixed ambient $g^{2t}_{2t(u+1)}$ of word type $\ubr{(12 \cdots (s+1))}_{\text{u times}}$ on $X$ is given by
\[
{\rm N}_1(t,u)= \sum_{k=0}^{d-2} {\rm N}_k^s(d) \binom{d+u-2-k}{d}
\]
where $d=t+1$, $s=2t$, and
\[
{\rm N}_k^s(d) := \#\{\text{top-to-bottom paths in }\Ga^s(d) \text{ with precisely }k \text{ backwards edges}\}.
\]
\end{thm}
\begin{rem}
As top-to-bottom paths in $\Ga^s(d)$ are in bijection with elements of the set of $d$-tuples
\[
{\rm W}^s(d):= \{(x_1,\dots,x_d): j \leq \sum_{i=1}^j x_i \leq s+j-1 \text{ for all } j=1. \dots, d\}
\]
it follows that ${\rm N}_k^s(d)$ is equal to the number of elements of $W^s(d)$ with precisely $k$ negative entries.
\end{rem}
In general, it is unclear if closed formulas for the coefficients ${\rm N}_k^s(d)$ exist. However, they are relatively straightforward to compute, as we will see shortly.

\begin{ex} Say $s=4$ and $d=3$. In this case, the associated diagram $\Ga^4(3)$ is as follows:

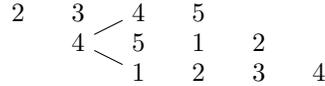
\hspace{90pt}
\begin{tikzpicture}[scale=.4]
\node (n1) at (1,10) {2};
\node (n2) at (3,10) {3};
\node (n3) at (5,10) {4};
\node (n4) at (7,10) {5};

\node (n5) at (3,9) {4};
\node (n6) at (5,9) {5};
\node (n7) at (7,9) {1};
\node (n8) at (9,9) {2};

\node (n9) at (5,8) {1};
\node (n10) at (7,8) {2};
\node (n11) at (9,8) {3};
\node (n12) at (11,8) {4};

\draw (n3)--(n5);
\draw (n5)--(n9);
\end{tikzpicture}

Top-to-bottom paths in the diagram with backwards edges linking vertices in the first and second rows correspond to the words $44m$, $54m$, and $55m$, where $1 \leq m \leq 4$; the traversal corresponding to $441$ is drawn. There are 12 such top-to-bottom paths; by symmetry, it follows there are 24 top-to-bottom paths containing at least one backwards edge in the diagram. Since there are $4^3=64$ top-to-bottom paths in total, there are 40 top-to-bottom paths without backwards edges. It follows that the total number of inclusions \eqref{basic_inclusion_r=1} in this case is $40 \binom{u+1}{3}+ 24 \binom{u}{3}$.
\end{ex}

\subsection{Top-to-bottom paths with prescribed numbers of backwards edges}
By working a bit harder, we can write explicit formulas for the number ${\rm N}^{2d-2}_j(d)$ of top-to-bottom paths in a $d \times (2d-2)$ grid $\Ga^{2d-2}(d)$ as above with precisely $k$ backwards edges. %For convenience's sake we rectify $\Ga^{2d-2}(d)$ to a $d \times (2d-2)$ rectangle, and count traversals with (precisely) $k$ backwards edges, where ``backwards" now means edges involving at least {\it two} horizontal backwards movements, instead of one.

To this end, we codify each top-to-bottom path as a binary string on $(d-1)$ bits, in which instances of 1 correspond to backwards edges. These form a poset $\mc{P}$ with $(d-1)$ levels, in which level $k$ consists of those vertices whose binary strings involve precisely $k$ instances of 1; there are $(d-1)$ levels because the minimal (resp., maximal) possible number of backwards edges in a top-to-bottom path is clearly 0 (resp., $(d-2)$).

The number of top-to-bottom paths with exactly the maximal number $k=d-2$ of backwards edges is easy to compute. Indeed, each such top-to-bottom path is classified by its associated binary string, which in turn describes which rows in the grid are associated with backwards edges. It is not hard to see in this case that the set of rows involving backwards edges is partitioned into at most two connected components, and that each component involving $j$ rows contributes a factor of ${\rm N}(j,2d-2j)$. It follows that
\begin{equation}\label{sum_k=d-2}
%\begin{split}
{\rm N}^{2d-2}_{d-2}(d)= \sum_{j=1}^{d-1} {\rm N}(j,2d-2j) \cdot N(d-j,2d-2(d-j))
%&= \sum_{j=1}^{d-1} {\rm N}(j,2d-2j) \cdot N(d-j,2j) \\
= \sum_{j=1}^{d-1}\binom{2d-j-1}{j} \binom{d+j-1}{d-j}.
%\end{split}
\end{equation}

Each of the products in \eqref{sum_k=d-2} represents the contribution of a binary string in $\mc{P}$. More generally, each vertex in $\mc{P}$ indexed by a string ${\bf w}$ has a natural {\it multiplicity} given by the number of top-to-bottom paths of a $d \times (2d-2)$ grid involving {\it at least} those backwards edges that are specified by ${\bf w}$; we may then use a process of inclusion-exclusion to determine the {\it exact} number of top-to-bottom paths involving (only) those backwards edges that are specified by ${\bf w}$.

Moreover, the multiplicity of a given string at level $k$ in $\mc{P}$ depends only on the underlying (unordered) partition $\la$ of $k$ to which it is associated; we denote this multiplicity by $m_d(\la)$. Explicitly, letting $\ell$ denote the length of $\la=(\la_1,\dots,\la_{\ell})$ (i.e. the total number of nonzero parts, which are allowed to be nondistinct) we have
\[
m_d(\la)= \prod_{i=1}^{\ell} {\rm N}(\la_i+1,2d-2\la_i) \cdot {\rm N}(1,2d-2)^{d-\ell-|\la|}
\]
where by convention ${\rm N}(1,2d-2)^{d-\ell-\|\la\|}=0$ whenever the exponent $d-\ell-|\la|$ is negative.

%There is one final combinatorial ingredient required 
In order to compute ${\rm N}^{2d-2}_j(d)$, another crucial combinatorial ingredient is the number $c_d(\la)$ of binary strings on $(d-1)$ bits associated to a given partition $\la$. Before deriving the general formula for $c_d(\la)$, we consider a typical situation in which all features of the general case are already present. Namely, say $\la=(1^j)$. We then want to count strings of type
\[
\ast 10 \ast 10 \ast \cdots \ast 10 \ast 1 \ast
\]
in which there are $d-2j$ asterisks that may either be empty or occupied by zeros. Accordingly, we see that $c_d((1^j))$ is equal to the coefficient of $x^{d-2j}$ in $(1+x+\dots+x^{d-2j})^{d-2j}$, from which it follows easily that
\[
c_d((1^j))= [x^{d-2j}](((1-x)^{-1})^{d-2j})= [x^{d-2j}](1-x)^{2j-d}= \bigg|\binom{2j-d}{d-2j}\bigg|.
\]
In the more general situation, write $\la=(\la_1^{e_1},\dots,\la_m^{e_m})$ where $\sum_{i=1}^m e_i=\ell$. We then have, accounting for reorderings of the (a priori unordered) partition $\la$:
\[
c_d(\la)= \binom{\ell}{e_1,\dots,e_m} [x^{d-|\la|-\ell}](1-x)^{\ell+|\la|-d}= \binom{\ell}{e_1,\dots,e_m} \cdot \bigg|\binom{\ell+|\la|-d}{d-|\la|-\ell}\bigg|.
\]
%{\bf Note: Should check against sign errors in generalized binomials (all of which should be positive.)}

%It follows that for every $0 \leq j \leq d-2$, the number of traversals involving {\it at least} $j$ backwards edges is
\noindent Now set
\begin{equation}\label{at_least_j}
\begin{split}
{\rm N}^{2d-2}_{j,+}(d)&:= \sum_{\la: |\la|=j} \binom{\ell}{e_1,\dots,e_m} \cdot \bigg| \binom{\ell+|\la|-d}{d-|\la|-\ell} \bigg| \cdot \prod_{i=1}^{\ell} {\rm N}(\la_i+1,2d-2\la_i) \cdot {\rm N}(1,2d-2)^{d-\ell-|\la|} \\
&= \sum_{\la: |\la|=j} \binom{\ell}{e_1,\dots,e_m} \cdot \bigg| \binom{\ell+|\la|-d}{d-|\la|-\ell} \bigg| \cdot (2d-2)^{d-\ell-|\la|} \cdot \prod_{i=1}^{\ell} \binom{2d- \la_i+1}{\la_i+1}.
\end{split}
\end{equation}
It follows from the discussion above that the number of top-to-bottom paths in $\Ga^{2d-2}(d)$ involving exactly $j$ backwards edges is given by
\begin{equation}\label{exactly_j}
{\rm N}^{2d-2}_{j}(d)= \sum_{j \leq k \leq d-2} (-1)^{k-j} \ga_j^k(d) \cdot {\rm N}^{2d-2}_{k, +}(d)
\end{equation}
where the coefficients $\ga_j^k(d)$ are positive integers specified by inclusion-exclusion carried out over $\mc{P}$. Explicitly, we have $\ga_j^j(d)=1$, and whenever $k>j$, the value of $\ga_j^k(d)$ is prescribed by the requirement that
\begin{equation}\label{gamma_equation}
\sum_{j \leq k \leq d-2} (-1)^{k-j} \binom{k}{j}\ga_j^k(d)=0.
\end{equation}

\subsection{Comparison with Macdonald's formula}
An inclusion \eqref{basic_inclusion_r=1} of linear series on a {\it smooth} curve $C$ codifies a {\it $d$-secant $(d-2)$-plane} to the image of $C$ in $\mb{P}^{2d-2}$. %In \cite{Co1}, we used Macdonald's formula to compute the generating series for the numbers $N_d=N_d(g,m)$ of these: explicitly, we have
%\begin{equation}\label{generating_series}
%\sum_{d \geq 0} N_d(g,m)= \mbox{exp} \bigg(\sum_{n>0} \fr{(-1)^{n-1}}{n} \bigg[\binom{2n-1}{n-1}m+ \bigg(4^{n-1}- \binom{2n-1}{n-1} \bigg) (2g-2) \bigg]z^n \bigg).
%\end{equation}
The (virtual) number $N_d=N_d(g,m)$ of these is computable via a classical formula of Macdonald's. %It is natural to compare our formula against the prediction of MacDonald's formula, which states the following:\
\begin{prop}[\cite{ACGH}, Proposition VIII.4.2]
The virtual number of included $\g^{s-d+r}_{m-d}$ inside a fixed $\g^s_m$ on a general genus-$g$ curve is 
\begin{equation}\label{macdonald_formula}
\begin{split}
N(g,s,m,d,r)&=\frac{(-1)^{r \choose 2}}{r!}\prod_{i=1}^r\Big[(1-t_i)^{g+s-m}\Big(1+\sum_i t_i \Big)^g\Delta(t)^2 \Big]_{(t_1t_2...t_r)^{s-d+2r}} \\
&=\frac{(-1)^{r \choose 2}}{r!} \Bigg[\prod_{i=1}^{s-d+r+1} (1+t_i)^{m-g-s} \Big(1+\sum_i t_i \Big)^g \Delta(t)^2 \Bigg]_{(t_1t_2...t_{s-d+r+1})^{s-d+2r}}
\end{split}
\end{equation}
where $\Delta(t)=\prod_{i>j}(t_i-t_j)$ and $[F(t_1,...,t_n)]_{t^{\vec{a}}}$ denotes the coefficient of $t^{\vec{a}}$ in the multi-variable polynomial $F(t_1,...,t_n)$. 
\end{prop}

%When $\rho(g,s,m)=\mu(d,r,s)=0$ and 
Note that in the statement of Proposition VIII.4.2 in {\it loc. cit.} there is a misprint in the second version of the formula, and the $s+d+2r$ appearing in the subscript should be replaced by $s-d+2r$. When $r=1$, the first version of Macdonald's formula specializes to the statement that 
\begin{equation}\label{macdonald_coefficients}
%N((2d+1)u,2d,2d(u+1),d,1)=\sum_{i=0}^{d+2}(-1)^i\binom{u}{i}\binom{(2d+1)u}{d+2-i}
N_d(g,m)= \sum_{i=0}^d(-1)^i\binom{g+2d-2-m}{i} \binom{g}{d-i}.
\end{equation}
%where $u=g+s-m$.

Of course, when $\rho=\mu=0$, the parameters $d$, $g$, and $m$ specialize to $d=t+1$, $g=(2t+1)u$, and $m=2t(u+1)$. Applying our smoothing theorem~\ref{thm:smoothing theorem} for good inclusions, together with Proposition~\ref{points_of_attachment_avoidance} (``all inclusions are good"), we immediately deduce the following result. %All available evidence supports the following.
\begin{thm}\label{equality_of_numbers}
Assume that $r=1$ and $\rho=\mu=0$, so that $d=t+1$, $g=(2t+1)u$, and $m=2t(u+1)$. The Macdonald numbers $N_d(g,m)$ of \eqref{macdonald_coefficients} agree with the numbers ${\rm N}_1(t,u)$ of Theorem~\ref{good_inclusions_r=1}, i.e. we have
\begin{equation}\label{combinatorial_equality}
\sum_{k=0}^{d-2} {\rm N}_k^{2d-2}(d) \binom{d+u-2-k}{d}=\sum_{i=0}^d(-1)^i\binom{u}{i}\binom{(2d-1)u}{d-i}.
\end{equation}
where the coefficients ${\rm N}_k^{2d-2}(d)$ are determined by equations~\ref{at_least_j} and \ref{exactly_j}.
\end{thm}
%Note that in light of our smoothing theorem~\ref{thm:smoothing theorem} for good inclusions, Conjecture~\ref{equality_of_numbers} holds if and only if the {\it only} inclusions \eqref{basic_inclusion_r=1} associated with an ambient $g^{2t}_{2t(u+1)}$ on a general chain of $(2t+1)u$ elliptic curves $X$ are good.

\begin{rem}
%Assuming that Conjecture~\ref{equality_of_numbers} holds, 
From Theorem~\ref{equality_of_numbers} it follows that Macdonald's formula, together with Theorem~\ref{combinatorial_equality}, determines the coefficients ${\rm N}_j^{2d-2}(d)$ {\it uniquely}. To see why this is interesting, note that in \cite{Co1} we computed the exponential generating function for Macdonald's secant plane numbers via intersection theory on the $d$th Cartesian product of a smooth curve. In that approach, the class formula is modeled on the combinatorics of the complete graph on $d$ vertices. So Theorem~\ref{equality_of_numbers} gives a bridge between the combinatorics of complete graphs and that of the graphs $\Ga(d,e)$.
\end{rem}

\subsection{Calculation of secant plane numbers ${\rm N}_1(t,u)$ for small values of $d$}\label{secant_plane_numbers_r=1}
Here we compute the coefficients ${\rm N}_k^{2d-2}(d)$ for all $2 \leq d \leq 6$, $k=0,\dots,d-2$ by explicitly applying the formulas of the preceding subsection.
\begin{itemize}
    \item {\bf $d=2$.} We have ${\rm N}_0^2(2)=4$, so \eqref{combinatorial_equality} reduces to the statement that
    \[
    4\binom{u}{2}=2u^2-2u.
    %\binom{3u}{2}- u(3u)+ \binom{u}{2}.
    \]
    \item {\bf $d=3$.} We have ${\rm N}_1^4(3)=24$ and ${\rm N}_0^4(3)= 4^3- {\rm N}_1^4(3)= 40$, so \eqref{combinatorial_equality} reduces to
    \[
    40 \binom{u+1}{3}+ 24 \binom{u}{3}= \fr{32}{3}u^3-12u^2+ \fr{4}{3}u.
    %\binom{5u}{3}- u\binom{5u}{2}+ \binom{u}{2}(5u)- \binom{u}{3}.
    \]
    \item {\bf $d=4$.} Applying \eqref{sum_k=d-2}, we find ${\rm N}_2^6(4)=148$. We next compute ${\rm N}_1^6(4)$ by applying the inclusion-exclusion relation \eqref{exactly_j}, which yields
{\small
\[{\rm N}_1^6(4)= {\rm N}_{1,+}^6(4)- 2\cdot {\rm N}_{2,+}^6(4)= 3 \cdot {\rm N}(2,4) \cdot 6^2- 2 \cdot(2 \cdot {\rm N}(3,2) \cdot 6+ {\rm N}(2,4)^2)= 784.\]
   }
   %Accordingly, we deduce that ${\rm N}_1^6(4)=932-148=784$ 
   %The sets on the right-hand side of \eqref{1_backwards_decomposition} comprise traversals of $\Ga^6(4)$ involving at least one backwards edge associated with particular bit sequences, and asterisks describe ambiguities. 
   It follows that ${\rm N}_0^6(4)=6^4-932=364$ and \eqref{combinatorial_equality} reduces to the statement that
    \[
    364 \binom{u+2}{4}+ 784 \binom{u+1}{4}+ 148 \binom{u}{4}= 54u^4-72u^3+20u^2-2u.
    \]
    \item $d=5$. We will check by explicit calculation that
    {\small
    \[
    %{\rm N}_0^5(8) \binom{u+3}{5}+ {\rm N}_1^5(8) \binom{u+2}{5}+ {\rm N}_2^5(8) \binom{u+1}{5}+ {\rm N}_3^5(8) \binom{u}{5}
    \sum_{j=0}^3 {\rm N}_j^5(8) \binom{u+3-j}{5}= \fr{4096}{15}u^5- \fr{1280}{3}u^4+ \fr{556}{3}u^3- \fr{100}{3}u^2+ \fr{8}{5}u.
    \]
    }
    To do so, we first apply \eqref{sum_k=d-2}, obtaining
   % {\small
    %\[
%{\rm N}_3^5(8)= \#{\rm S}_{0001}+ \#{\rm S}_{1000}+ \#{\rm S}_{0100}+ \#{\rm S}_{0010}
%\#{\rm S}_{000 \ast \ast}+ \#{\rm S}_{\ast 000 \ast }+\#{\rm S}_{\ast \ast 000}+\#{\rm S}_{0\ast00\ast} +\#{\rm S}_{0\ast \ast 00}+ +\#{\rm S}_{\ast 0\ast 00}+\#{\rm S}_{00\ast0\ast} +\#{\rm S}_{00\ast \ast 0}+ +\#{\rm S}_{\ast 00\ast 0}
%= 2 \cdot {\rm N}(4,2) \cdot 8+ 2 \cdot {\rm N}(3,4) \cdot {\rm N}(2,6)=1000. 
%    \]
%    }
${\rm N}_3^5(8)= 920$. Next, applying \eqref{exactly_j}, we compute
    {\small
    \[
    \begin{split}
    {\rm N}_2^5(8)&= {\rm N}_{2,+}^5(8)- 3 \cdot {\rm N}_{3,+}^5(8)=
    %\#{\rm S}_{00\ast \ast}+ \#{\rm S}_{\ast 00 \ast}+ \#{\rm S}_{\ast \ast 00}+ \#{\rm S}_{0 \ast 0 \ast}+ \#{\rm S}_{0 \ast \ast 0}+ \#{\rm S}_{\ast 0\ast 0}- 3(\#{\rm S}_{000\ast}+ \#{\rm S}_{00 \ast 0}+\#{\rm S}_{0\ast 00}+ \#{\rm S}_{\ast 000}) \\
    3 \cdot {\rm N}(3,4) \cdot 8^2+ 3 \cdot {\rm N}(2,6)^2 \cdot 8- 3 \cdot {\rm N}_3^5(8)=11664 \text{ and}\\
    {\rm N}_1^5(8)&={\rm N}_{1,+}^5(8)-2 \cdot {\rm N}_{2,+}^5(8)+ 3 \cdot {\rm N}_{3,+}^5(8)=4 \cdot {\rm N}(2,6) \cdot 8^3- 2 (3 \cdot {\rm N}(3,4) \cdot 8^2+ 3 \cdot {\rm N}(2,6)^2 \cdot 8)+ 3 \cdot {\rm N}_3^5(8)=16920;
    \end{split}
    \]
    }
    %In similar fashion, we compute
%{\small
%\[
%\begin{split}
%{\rm N}_1^5(8)&=\#{\rm S}_{0\ast\ast\ast}+ \#{\rm S}_{\ast 0\ast\ast}+\#{\rm S}_{\ast\ast 0\ast}+ \#{\rm S}_{\ast\ast\ast 0}-2(\#{\rm S}_{00\ast \ast}+ \#{\rm S}_{\ast 00 \ast}+ \#{\rm S}_{\ast \ast 00}+ \#{\rm S}_{0 \ast 0 \ast}+ \#{\rm S}_{0 \ast \ast 0}+ \#{\rm S}_{\ast 0\ast 0}) \\
%&+3(\#{\rm S}_{000\ast}+ \#{\rm S}_{00 \ast 0}+\#{\rm S}_{0\ast 00}+ \#{\rm S}_{\ast 000}) \\
%&= 4 \cdot {\rm N}(2,6) \cdot 8^3- 2 (3 \cdot {\rm N}(3,4) \cdot 8^2+ 3 \cdot {\rm N}(2,6)^2 \cdot 8)+ 3 \cdot {\rm N}_3^5(8) \\
%&=16920;
%\end{split}
%\]
%}
\vspace{-5pt}
it follows that ${\rm N}_0^5(8)= 8^5- ({\rm N}_1^5(8)+ {\rm N}_2^5(8)+{\rm N}_3^5(8))= 3264$.

\medskip
\item $d=6$. We will explicitly check that
    {\small
    \[
    %{\rm N}_0^6(10) \binom{u+4}{6}+ {\rm N}_1^6(10) \binom{u+3}{6}+ {\rm N}_2^6(10) \binom{u+2}{6}+ {\rm N}_3^6(10) \binom{u+1}{6}+ {\rm N}_4^6(10) \binom{u}{6}
    \sum_{i=0}^4 {\rm N}_j^6(10) \binom{u+4-i}{6}= \fr{12500}{9}u^6- 2500u^5+ \fr{13100}{9}u^4- 386u^3+ \fr{392}{9}u^2-2u.
    \]
    }
Applying \eqref{sum_k=d-2} yields ${\rm N}_4^6(10)=5776$. Next via \eqref{exactly_j} we compute
{\small
\[
\begin{split}
{\rm N}_3^6(10) &= 3 \cdot {\rm N}_{3,+}^6(10)- 4 \cdot {\rm N}_{3,+}^6(10)=155012; \\ %3 \cdot \#{\rm S}_{000\ast \ast}+ 6 \cdot \#{\rm S}_{0 \ast \ast 00}+ \#{\rm S}_{0 \ast 0 \ast 0}- 4 \cdot {\rm N}_4^6(10)= 155012; \\
%\#{\rm S}_{000\ast \ast}+ \#{\rm S}_{\ast 000\ast}+ \#{\rm S}_{\ast \ast000}+ \#{\rm S}_{0\ast \ast 00}+ \#{\rm S}_{0\ast 00 \ast}+ \#{\rm S}_{\ast 0\ast 00}+ \#{\rm S}_{0\ast 0 \ast 0}+ \#{\rm S}_{00\ast \ast 0}+ \#{\rm S}_{00\ast 0 \ast}+ \#{\rm S}_{\ast 00\ast 0}-4 \cdot {\rm N}_4^6(10) \\
%(\#{\rm S}_{0000\ast}+ \#{\rm S}_{000\ast 0}+ \#{\rm S}_{00\ast 00}+ \#{\rm S}_{0\ast 000}+ \#{\rm S}_{\ast 0000})
{\rm N}_2^6(10) &= 4 \cdot {\rm N}_{2,+}^6(10)- 3\cdot {\rm N}_{3,+}^6(10)+ 6 \cdot {\rm N}_{4,+}^6(10); \text{ and} \\%= 4 \cdot \#{\rm S}_{00\ast \ast \ast}+ 6 \cdot \#{\rm S}_{0 \ast 0 \ast \ast}- 3 \cdot (3 \cdot \#{\rm S}_{000\ast \ast}+ 6 \cdot \#{\rm S}_{0 \ast \ast 00}+ \#{\rm S}_{0 \ast 0 \ast 0})+ 6 \cdot {\rm N}_4^6(10)=501908; \text{ and}\\
{\rm N}_1^6(10) &= 5 \cdot {\rm N}_{1,+}^6(10)- 2 \cdot {\rm N}_{2,+}^6(10)+ 3 \cdot {\rm N}_{3,+}^6(10)- 4 \cdot {\rm N}_{4,+}^6(10)=308044. %=5 \cdot \#{\rm S}_{0 \ast\ast\ast\ast}- 2 \cdot (4 \cdot \#{\rm S}_{00\ast \ast \ast}+ 6 \cdot \#{\rm S}_{0 \ast 0 \ast \ast})+ 3 \cdot (3 \cdot \#{\rm S}_{000\ast \ast}+ 6 \cdot \#{\rm S}_{0 \ast \ast 00}+ \#{\rm S}_{0 \ast 0 \ast 0})- 4 \cdot {\rm N}_4^6(10)=308044.
\end{split}
\]
}
It follows that ${\rm N}_0^6(10)= 10^6-({\rm N}_1^6(10)+{\rm N}_2^6(10)+{\rm N}_3^6(10)+{\rm N}_4^6(10))=29260$.
\end{itemize}

\subsection{The case $r=s-1$}
In this case, solving $\rho=\mu=0$ yields $d=2r$, $s=r+1$, $g=(r+2)u$, and $m=(r+1)(u+1)$; rewritten in terms of $r$ and $u$, the basic inclusion \eqref{basic_inclusion} becomes
\begin{equation}\label{basic_inclusion_r=s-1}
g^1_{(r+1)(u+1)-2r}+ p_1+ \cdots+ p_{2r} \hra g^{r+1}_{(r+1)(u+1)}.
\end{equation}
The combinatorics in this regime is more complicated. Its distinguishing feature is that the shifting of the distinguished indices (associated to the included limit linear pencil) induced by the presence of base points is not deterministic as in the $r=1$ case; rather, shifting exhibits {\it branching}, in a way that is explicitly predicted by the Pl\"ucker poset of $\text{Gr}(s-d+r+1,s+1)=\text{Gr}(2,r+2)$. More precisely, the set of Pl\"ucker coordinates of $\text{Gr}(2,r+2)$ is indexed by pairs of numbers from the index set $[r+2]=\{1,\dots,r+2\}$. On the other hand, each component of our elliptic chain that contains a base point of the included linear series is associated to a {\it pair of pairs} $\pi_1,\pi_2 \in [r+2]^2$ which share a common edge in the Pl\"ucker poset. In particular, $\pi_1$ and $\pi_2$ have an element in common in $[r+2]$. The important point is now the following.

\begin{lem}\label{local_obstruction_r=s-1} Assume the ambient series $g^{r+1}_{(r+1)(u+1)}$ is of word type $\ubr{(12 \cdots (s+1))}_{\text{u times}}$, and that $(\pi_1,\pi_2)$ is an edge of the Pl\"ucker poset for $\text{Gr}(2,r+2)$, with $\nu= \pi_1 \cap \pi_2$. Every index $j=j(\pi_1,\pi_2)$ of a component of a general elliptic chain along which an inclusion \eqref{basic_inclusion_r=s-1} has a base point whose local evolution (i.e., shifting) is described by $(\pi_1,\pi_2)$ satisfies $j \not\equiv \nu \text{ (mod }(s+1))$.
\end{lem}

\begin{proof} The index $\nu$, by definition, describes the unique aligned section of the aspect of the included limit linear pencil along the $j$th component. On the other hand, because the ambient series $g^{r+1}_{(r+1)(u+1)}$ is of word type $\ubr{(12 \cdots (s+1))}_{\text{u times}}$, $j \mod (s+1)$ is the distinguished index of the ambient aspect $g^s_m$ associated with the unique section that vanishes to total order $m$ in the points of attachment of the $j$th component. It follows immediately that $j \not\equiv \nu \text{ (mod }(s+1))$.
\end{proof}
%\begin{conj}The condition $j \not\equiv \rho \text{ (mod }s+1)$ is the {\it only} prohibited index in the situation described by the Lemma.
%\end{conj}

%Assuming the preceding conjecture holds (and presumably it isn't hard to prove!) 
It is not hard to see, moreover, that $j$ being congruent to $\nu$ modulo $(s+1)$ is the {\it only} local numerical obstruction to the placement of a base point along the $j$th component of our elliptic chain. Accordingly, counting good inclusions in the $r=s-1$ case reduces to resolving a graphical enumeration problem analogous to the one described in the $r=1$ case for {\it each top-to-bottom traversal of the Pl\"ucker poset} $\mc{P}$, and then summing over all traversals of the Pl\"ucker poset in order to obtain the total number of inclusions \eqref{basic_inclusion_r=s-1}. 

\medskip
Explicitly, each traversal of $\mc{P}$ is indexed by a {\it prohibition sequence} $\La$ comprised of the $d$ forbidden base point indices in Lemma~\ref{local_obstruction_r=s-1}. To each fixed choice of $\La$ we associate a $d \times g$ grid ${\rm G}(\La)$ whose $j$th row is of the form $\ubr{(12 \cdots (s+1))}_{\text{u times}}$, but in which no entry $\La(j)$ may be traversed. From the prohibition grid ${\rm G}(\La)$ we next extract a diagram $\Ga^s(d;\La)$ whose traversals with prescribed numbers $k$ of backwards edges (nearly) stratify the space of $d$-tuples on the alphabet $[s]$ according to their associated $u$-binomial contributions. The total number of positive traversals of ${\rm G}(\La)$ is a sum of (multiples of) $u$-binomials $\binom{u+d-2-k}{2}$, $k=0, \dots, d-2$, plus a finite leftover term that we label ${\rm R}_{\La}$. Summing over all top-to-bottom traversals $\Ga$ of $\mc{P}$ yields the following result.

\begin{thm}\label{good_inclusions_r=s-1} When $\rho=\mu=0$ and $r=s-1$, the number of inclusions \eqref{basic_inclusion_r=s-1} for a fixed ambient $g^{r+1}_{(r+1)(u+1)}$ of combinatorial word type $\ubr{(12 \cdots (s+1))}_{\text{u times}}$ on $X$ is given by
\[
{\rm N}_{s-1}(t,u)= \sum_{k=0}^{d-2} {\rm N}_k^s(d) \binom{d+u-2-k}{d} + {\rm R}
\]
in which $d=2r$, $s=r+1$,
\[
{\rm N}_k^s(d) := \sum_{\La} \#\{\text{traversals of }\Ga^s(d; \La) \text{ with precisely }k \text{ backwards edges}\} \text{ and } {\rm R}:= \sum_{\La} {\rm R}_{\La}.
\]
Here $\La$ varies over all top-to-bottom traversals of the Pl\"ucker poset $\mc{P}$ of $\mbox{Gr}(2,s+1)$.
\end{thm}

\subsubsection{Example}
Say $d=6$ and $s=4$. In this case we are counting inclusions
\[
g^1_{4u-2} + p_1+ \dots+ p_6  \hra g^4_{4u+4}
\]
along a chain of genus $g=5u$. Here the relevant Grassmannian is $\text{Gr}(2,5)$, whose Pl\"ucker poset is drawn below.

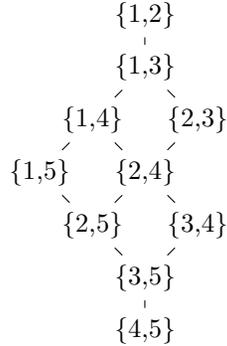
\hspace{150pt}
\begin{tikzpicture}[scale=.7]
\node (n1) at (3,10) {\{1,2\}};
\node (n2) at (3,9) {\{1,3\}};

\draw (n1)--(n2);

\node (n3) at (2,8) {\{1,4\}};
\node (n4) at (4,8) {\{2,3\}};

\draw (n2)--(n3);
\draw (n2)--(n4);

\node (n5) at (1,7) {\{1,5\}};
\node (n6) at (3,7) {\{2,4\}};

\draw (n3)--(n5);
\draw (n3)--(n6);
\draw (n4)--(n6);

\node (n7) at (2,6) {\{2,5\}};
\node (n8) at (4,6) {\{3,4\}};

\draw (n5)--(n7);
\draw (n6)--(n7);
\draw (n6)--(n8);

\node (n9) at (3,5) {\{3,5\}};
\node (n10) at (3,4) {\{4,5\}};

\draw (n7)--(n9);
\draw (n8)--(n9);
\draw (n9)--(n10);

\end{tikzpicture}

%i.e. sequences which specify which indices $k(j)$ of components of the elliptic chain along which the $j$th base point {\it cannot} occur, where $1 \leq j \leq 6$. To each prohibition sequence we naturally associate a $d \times g$ grid $\Ga^s(d;\La)$ with rows of the form $\ubr{(12345)}_{u \text{ times}}$ in which the $j$th entry of the prohibition sequence is a prohibited position in the $j$th row. %For example, when $u=4$, the prohibition sequence $(1,3,2,4,3,5)$ is associated with the prohibition grid below.
%\[
%\begin{array}{cccccccccccccccccccccc}
%\ast & 2 & 3 & 4 & 5 & \ast & 2 & 3 & 4 & 5 & \ast & 2 & 3 & \ast & \ast & \ast & \ast & \ast & \ast & \ast \\
%\ast & \ast & \ast & 4 & 5 & 1 & 2 & \ast & 4 & 5 & 1 & 2 & \ast & 4 & \ast & \ast & \ast & \ast & \ast & \ast \\
%\ast & \ast & \ast & \ast & 5 & 1 & \ast & 3 & 4 & 5 & 1 & \ast & 3 & 4 & 5 & \ast & \ast & \ast & \ast & \ast \\
%\ast & \ast & \ast & \ast & \ast & 1 & 2 & 3 & \ast & 5 & 1 & 2 & 3 & \ast & 5 & 1 & \ast & \ast & \ast & \ast \\
%\ast & \ast & \ast & \ast & \ast & \ast & 2 & \ast & 4 & 5 & 1 & 2 & \ast & 4 & 5 & 1 & 2 & \ast & \ast & \ast \\
%\ast & \ast & \ast & \ast & \ast & \ast & \ast & 3 & 4 & \ast & 1 & 2 & 3 & 4 & \ast & 1 & 2 & 3 & 4 & \ast
%\end{array}
%\]

%From the prohibition grid associated with a fixed choice of prohibition sequence and generic choice of $u$ we extract a diagram $\Ga^s(d;\La)$ whose traversals stratify the space of $d$-tuples on the alphabet $[s]$ according to their associated $u$-binomial contributions. We then sum up these individual contributions.
There are five top-to-bottom traversals of the Pl\"ucker poset $\mc{P}$ of $\mbox{Gr}(2,5)$, corresponding to five prohibition sequences $\La$. A typical traversal of $\mc{P}$ is $\La=(1,3,2,4,3,5)$. The corresponding finite diagram $\Ga^4(6;(1,3,2,4,3,5))$ is given by
{\small
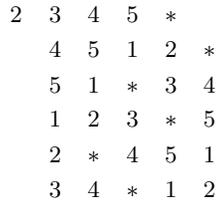
\[
\begin{array}{cccccc}
2 & 3 & 4 & 5 & \ast & \empty \\
\empty & 4 & 5 & 1 & 2 & \ast \\
\empty & 5 & 1 & \ast & 3 & 4 \\
\empty & 1 & 2 & 3 & \ast & 5 \\
\empty & 2 & \ast & 4 & 5 & 1 \\
\empty & 3 & 4 & \ast & 1 & 2
\end{array}
\]
}
and the coefficient ${\rm N}^s_k(d; (1,3,2,4,3,5))$ of $\binom{d+u-2-k}{d}$ contributed by prohibition sequence $(1,3,2,4,3,5)$ is equal to the number of top-to-bottom traversals of $\Ga^4(6)$ with (exactly) $0 \leq k \leq d-2$ backwards arrows, plus {\bf 4}, which is the number of top-to-bottom paths (in which we visit precisely one vertex in each row) through the following leftover diagram corresponding to the $u=2$ case:
{\small
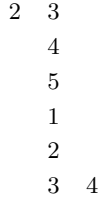
\[
\begin{array}{ccc}
2 & 3 \\
\empty & 4 \\
\empty & 5 \\
\empty & 1 \\
\empty & 2 \\
\empty & 3 & 4
\end{array}.
\]
}
%{\bf \flu Case one: Prohibition sequence $(1,3,2,4,3,5)$.} Just as in the $r=1$ case, we may naturally associate a $d \times g$ grid with rows of the form $\ubr{(12345)}_{u \text{ times}}$ in which the $j$th entry of the prohibition sequence is a prohibited position in the $j$th row. We then count all possible positive traversals of this {\it prohibition grid}. We illustrate the case $u=4$ below, with additional prohibitions (marked by asterisks) arising from monotonicity at the boundary already included.

%From the prohibition grid, we next extract a diagram $\Ga^s(d)$ whose traversals serve to stratify the space of $d$-tuples on the alphabet $[s]$ according to their associated $u$-binomial contributions. 
\subsection{Comparison with Macdonald's formula}
When $r=s-1$, the second version of Macdonald's formula ~\ref{macdonald_formula} establishes that the virtual number of $2r$-secant $(r-2)$-planes to a smooth curve in $\mb{P}^{r-1}$ is given by
\begin{equation}\label{macdonald_r=s-1}
\frac{(-1)^{r \choose 2}}{r!} \Big[((1+t_1)(1+t_2))^{m-g-s} (1+t_1+t_2)^g (t_1-t_2)^2 \Big]_{t_1^{r+1}t_2^{r+1}}.
\end{equation}
When $\rho=\mu=0$, so that $g=u(s+1)$ and $m=(r+1)(u+1)$, Theorem~\ref{thm:smoothing theorem} and Proposition~\ref{points_of_attachment_avoidance} imply that the expression in \eqref{macdonald_r=s-1} agrees precisely with our good inclusion number ${\rm N}_{s-1}(t,u)$.

\begin{thm}\label{combinatorial_equality_r=s-1}
Assume $\rho=\mu=0$, and let ${\rm N}_{s-1}(t,u)$ denote the number of inclusions \eqref{basic_inclusion_r=s-1} on a general elliptic chain of genus $g=u(s+1)$ for which the ambient series is of combinatorial type $\ubr{(12 \cdots (s+1))}_{\text{u times}}$. We have
\[
{\rm N}_{s-1}(t,u)=\frac{(-1)^{r \choose 2}}{r!} \Big[((1+t_1)(1+t_2))^{-u} (1+t_1+t_2)^{u(r+2)} (t_1-t_2)^2 \Big]_{t_1^{r+1}t_2^{r+1}}
\]
where ${\rm N}_{s-1}(t,u)$ is as defined in the statement of Theorem~\ref{good_inclusions_r=s-1}.
\end{thm}

%Like Conjecture~\ref{equality_of_numbers}, Conjecture~\ref{combinatorial_equality_r=s-1} predicts that all inclusions are good when the ambient series is of the combinatorially most-favorable type.

\subsection{The general case}
Counting included series inside a fixed ambient series $g^s_m$ of type $\ubr{(12 \cdots (s+1))}_{\text{u times}}$ when $\rho=\mu=0$ and $1 \leq r \leq s-1$ is arbitrary follows the basic template established in the discussion of the $r=s-1$ case; the caveat is that one must replace the Pl\"ucker poset by the appropriate subposet determined by the collection of possible shift sequences induced by the placement of base points along the (interiors of) components of a general elliptic chain curve $X$. In this way we obtain an obvious generalization of Theorems ~\ref{equality_of_numbers} and \ref{combinatorial_equality_r=s-1}.

\begin{rem}
The monodromy action on $G^s_m(C)$ induced by a degeneration of a general genus-$g$ curve $C$ to a flag curve is well-known to be transitive whenever $\rho(g,s,m)=0$ by a celebrated result of Eisenbud and Harris; it follows that the number of $d$-secant $(d-r-1)$-planes to (the image of) each $g^s_m$ is constant whenever $\mu=0$. In particular, the {\it total number} of $d$-secant $(d-r-1)$-planes to series $g^s_m$ on a general curve is simply $\eta$ times the Macdonald number $N(g,s,m,d,r)$, where $\eta=\eta(g,s,m)$ is the generalized Catalan number
\[
\eta= g! \cdot \prod_{i=0}^s \fr{i!}{(g-m+s+i)!}.
\]
Our results thus show that when $\rho=\mu=0$ all secant planes on a general curve are ``detected" by inclusions of limit linear series. A next logical line of inquiry would be to extend our counting scheme to compute counts of secant planes along a general curve in situations for which the {\it sum} $\rho+\mu$ is zero, but $\rho$ itself is positive.
%The fact that all {\it good} inclusions \eqref{basic_inclusion} smooth when the ambient $g^s_m$ is of most-favorable combinatorial type indicates, assuming that the natural generalization of Conjectures~\ref{equality_of_numbers} and \ref{combinatorial_equality_r=s-1} hold, that when $\rho=\mu=0$ all secant planes on a general curve are ``detected" by good inclusions of limit linear series.
\end{rem}

\section{A dimension-theoretic moduli pathology}\label{pathology}
What follows is an example, due to Melody Chan, that shows that the moduli space of included limit linear series constructed in Section~\ref{sec:moduli} may have unexpectedly large dimension. This already occurs near the points whose ambient limit linear series has word type distinct from that used in our smoothing theorem~\ref{thm:smoothing theorem}.

\begin{ex}\label{ex:Chan's example}
Let $(g,d_1,d_2,r_1,r_2)=(12,10,7,2,1)$. Here $\rho(g,r_1,d_1)=0 $ and $\mu=-1$, so we expect the corresponding moduli space of inclusions of limit linear  series to be empty. We will show that this is not the case in general. Indeed, let $E$ be a chain of 13 curves $Z_i$ in which all but the middle component are elliptic, and the remaining component $Z_7$ is rational. There is a {\it one-dimensional} family of inclusions of limit linear series with three base points supported on the rational component. In this family, the ambient limit linear series is constant and its aspects are specified by the following vanishing sequences at the nodes:

{\tiny
\[\begin{tabular}{cccc}
$Z_1\colon$&\underline 0&\underline 1&2\\&\underline{10}&\underline 8&7\\
\hline
$Z_2\colon$ &\underline 0&\underline 2&3\\&\underline{10}&\underline 7&6\\
\hline
 $Z_3\colon$ &   \underline 0&\underline 3&4\\&\underline{10}&\underline 6&5\\
\hline
$Z_4\colon$ &\underline 0&\underline 4&5\\&\underline 9&\underline 6&4\\
\hline
$Z_5\colon$&\underline 0&\underline 5&6\\&\underline 9&\underline 5&3\\
\hline
$Z_6\colon$&\underline 1&\underline 5&7\\&\underline 8&\underline 5&2
\\
\hline
$Z_7=\mathbb P^1_k\colon$ &\underline 2&\underline 5&8\\&8&\underline 5&\underline 2\\
\hline
$Z_8\colon$ &2&\underline 5&\underline 8\\&7&\underline 5&\underline 1\\
\hline
$Z_9\colon$& 3&\underline 5&\underline 9\\&6&\underline  5&\underline 0\\
\hline
$Z_{10}\colon$ &4&\underline 5&\underline{10}\\&5&\underline 4&\underline 0\\
\hline
$Z_{11}\colon $&5&\underline 6&\underline{10}\\&4&\underline 3&\underline 0\\
\hline
$Z_{12}\colon$ &6&\underline 7&\underline{10}\\&3&\underline 2&\underline 0\\
\hline
$Z_{13}\colon$&7&\underline 8&\underline{10}\\&2&\underline 1&\underline 0.
\end{tabular}\]
}
The included linear series on each elliptic component $Z_j$, $j\neq 7$, is specified unambiguously by the underlined vanishing orders, however there is a one-dimensional ambiguity along $Z_7$. To see this, let $(x,y)$ be projective coordinates on $Z_7$, and let $P=0=Z_6\cap Z_7$ and $Q=\infty=Z_7\cap Z_8$ be the two distinguished points of attachment on $Z_7$. %To obtain a inclusion of linear series with given vanishing sequences on $Z_7$ we can take the ambient series 
Letting $V^1_7:=\mathrm{span}\{x^8y^2,x^5y^5,x^2y^8\}$ and $V^2_7:=\mathrm{span}\{x^8y^2+\lambda x^5y^5,x^5y^5+\lambda x^2y^8\}$, we obtain an inclusion of linear series $V^2_7 \hra V^1_7$ for every $\lambda \neq 0$. So associated with every $\lambda\neq 0$ there is an inclusion of limit linear series with three base points $\{(\omega:1)|\omega^3=-\lambda\}$ on $Z_7$.
\end{ex}

\bibliographystyle{amsalpha}
\bibliography{myrefs2}
\end{document}